\documentclass[UTF8]{article}
\usepackage{amssymb,amsmath,amsthm,mathrsfs,multirow,xcolor,framed,url,authblk}
\usepackage{float}
\usepackage{lipsum}
\usepackage{tabularx}
\usepackage{booktabs}
\usepackage{longtable}

\usepackage{tikz}
\usepackage[all]{xy}
\usepackage[ 
bookmarks=true,
bookmarksnumbered=true
bookmarksopen=flase,
hypertexnames=false,
breaklinks=true,
colorlinks=true,
pagebackref=true,
hyperindex=true,
plainpages=false
pdfauthor={Zhou }
pdftitle={Ternary}
pdfsubject={ }
pdfkeywords={}
 ]
{hyperref}

\newif\ifpdfAuthoring
\pdfAuthoringtrue
\ifpdfAuthoring
\fi
\newtheorem{theorem}{Theorem}[section]
\newtheorem{lemma}[theorem]{Lemma}
\newtheorem{cor}[theorem]{Corollary}

\newtheorem{prop}[theorem]{Proposition}

\theoremstyle{definition}

\usepackage{tikz}
\usetikzlibrary{arrows.meta}
\theoremstyle{remark}
\newtheorem{remark}[theorem]{Remark}

\numberwithin{equation}{section}

\usepackage{hyperref}
\newcommand\nutwid{\overset {\text{\lower 3pt\hbox{$\sim$}}}\nu}

\newcommand{\Nodd}{N^{\text{odd}}}
\newcommand{\Neven}{N^{\text{even}}}
\newcommand{\tr}{\mbox{tr}}
\newcommand{\n}{\mbox{n}}
\newcommand{\discrd}{\mbox{discrd}}
\newcommand{\disc}{\mbox{disc}}

\allowdisplaybreaks

\providecommand{\keywords}[1]
{\par\hspace{\parindent}\small\textbf{\bfseries\centering Keywords :} #1 }
\providecommand{\MSC}[1]
{\par\hspace{\parindent}\small\textbf{\bfseries\centering MSC :} #1 }

\newcommand{\Aut}{\operatorname{Aut}}

\begin{document}

\title{The classification and representations of positive definite ternary quadratic forms of level $4N$}

\date{}
\author[1]{Yifan Luo}
\author[1,2]{Haigang Zhou}
\affil[1]{School of Mathematical Sciences, Tongji University, Shanghai, China, 200092}
\affil[2]{Institute of Mathematics, Henan Academy of Sciences,  Zhengzhou, Henan, China, 450046}




\maketitle

\tableofcontents

\begin{abstract}
  Classifications and representations are two main topics in the theory of quadratic forms. In this paper, we consider these topics of ternary quadratic forms. For a given squarefree integer $N$, first we give the classification of positive definite ternary quadratic forms of level $4N$ explicitly. Second, we give explicit formulas of the weighted sum of representations over each class in every genus of ternary quadratic forms of level $4N$, which are involved with modified Hurwitz class number. In the proof of the main results, we use the relations among ternary quadratic forms, quaternion algebras, and Jacobi forms. As a corollary, we get the formula for the class number of positive ternary quadratic forms of level $4N$.
 As applications, we derive an explicit base of Eisenstein series space of modular forms of weight $3/2$ and level $4N$, and give new proofs of some interesting identities involving representation number of ternary quadratic forms.
\end{abstract}
	
	


\MSC{11E20,  11R52, 11R29, 11F50, 11F37}

\section{Introduction}

Let $f$ be a ternary quadratic form with integer coefficients, given by the equation
\begin{equation*}
f(x,y,z)=ax^2+by^2+cz^2+ryz+sxz+txy.
\end{equation*}
Unless otherwise stated, we assume that $f$ is primitive and positive definite. We will also denote $f$ by $(a,b,c,r,s,t)$. Recall that the matrix associated to $f$ is
\begin{equation*}
M=M_f=
\begin{pmatrix}
2a & t & s\\
t & 2b & r\\
s & r & 2c
\end{pmatrix}
.
\end{equation*}
Define the discriminant of $f$ to be
\begin{equation*}
d=d_f=\frac{\det(M_f)}{2}=4abc+rst-ar^2-bs^2-ct^2. 
\end{equation*}
The level of $f$ is the smallest positive integer $N$ such that $NM_f^{-1}$ is even, that is, has integral entries, and even integers on the main diagonal. 
 Let the number of ways to represent  integer $n$ by ternary quadratic form $f$ be
$
	R_f(n) := 
	\sharp\{X \in \mathbb{Z}^3 \mid f(X)= \frac 12 X^tM_fX=n    \}. 
$

Classifications and representations are two main topics in the theory of quadratic forms. Regarding the classifications of ternary quadratic forms, two significant equivalent relations exist, namely equivalence class and semi-equivalence class. The set of semi-equivalence class is called genus. Extensive tables of positive definite ternary quadratic forms, categorized by discriminant, have been compiled. The tables compiled by Brandt and Intrau comprehensively document all reduced ternary forms with discriminant $d<1000$.  Lehman \cite{Leh92} grouped positive definite ternary quadratic forms differently by level. He gave some correspondences between classes of ternary quadratic forms having the same level with different discriminants and provided a practical method for finding representatives of all classes of ternary forms with a given level. However,  the method for determining whether two ternary quadratic forms are semi-equivalent is still cumbersome. In this paper, for squarefree interger $N$, we group primitive positive definite ternary quadratic forms of level $4N$ more explicitly, which only depends on level, discriminant, and the place where $f$ is anisotropic, as follows.

 In the following, we always assume that  $N$ is a product of $s$ distinct odd primes. Let $\Nodd$ (respectively, $\Neven$, or $N_r$) denote one of  divisors of $N$ with odd number of (respectively, even number of, or exactly  $r$ many  ) prime factors. We use $G_{4N,d,\Nodd}$ (resp. $G_{4N,d,2\Neven}$) for the genus of primitive positive definite ternary quadratic forms of level $4N$, discriminant $d$, and these forms are anisotropic only in the $p$-adic field where $p\mid\Nodd$ (resp.$p\mid 2\Neven$).

\begin{theorem}\label{theo:cla}
	Let $N$ be a product of $s$ distinct odd primes.  Then for all primitive positive definite ternary quadratic forms of level $4N$, there are $2^{2s+1}$ genera.  These are $G_{4N,N^2/N_r,\Nodd}$, $G_{4N, 4N^2/N_r,\Nodd}$, $G_{4N,4N^2/N_r,2\Neven}$ and $G_{4N,16N^2/N_r,\Nodd}$, where $\Nodd$ (respectively, $\Neven$ or $N_r$ ) runs over all divisors of $N$ that contain an odd number of (respectively, an even number of, or exactly $r$ many) prime factors.
\end{theorem}
 The representation problem is the questions  if an integer $n$ is represented by an integral quadratic form in $s$ variables and in how many ways  $n$ is represented by such an integral quadratic form.  The literature on  quadratic forms is extensive and highly developed.  Here we only focus on the case of ternary quadratic forms. Legendre gave   the necessary and sufficient condition of when $n$ is represented by a sum of three squares. Gauss went further, giving an explicit formula of the number of ways to represented $n$ by sum of three squares, which is involved with the class number of binary quadratic forms.  However, there are only a few ternary quadratic forms which have representation formulas like that of sum of three squares.
  Siegel provided a significant quantitative outcome in this regard by presenting  $ r_{\mathrm{gen}(Q)}(n)$,  a weighted average of representations through forms in the genus of  a quadratic form $Q(x)$, as an infinite product of local factors.  There are many literature  on representation of ternary quadratic forms based on Siegel-Weil formula and modular forms.    Recently, X.~Guo, Y.~Peng, L.~Gao and H.~Qin \cite{GPQ14}\cite{GaQ19} gave some explicit formulas for the  average number of representations over the genus of ternary quadratic form  of type $f=x^2+py^2+qz^2$, where $p$ and $q$ are odd primes.  B.~Kane, D.~Kim and S.~Varadharajan \cite{KKV23} computed explicitly  the Siegel–Weil average   $ r_{\mathrm{gen}(Q)}(n)$  of a genus for ternary quadratic forms corresponding to stable lattices.

 Orders of quaternion algebras have a close relation with ternary quadratic forms.  For the detail of the relation between ternary forms and orders of quaternion algebras, we refer to the book of J.~Voight \cite{Voi21}. 
In 1987, B.~Gross \cite{Gro87} showed that,  for definite quaternion algebra ramified at prime $p$, the weighted sum of theta series corresponding to maximal orders is Eisenstein series of weight $3/2$ whose Fourier coefficients are modified Hurwitz class number. In another words,  the modified Hurwitz class number equals a weighted sum of the number of elements of trace $0$ and norm $n$ in an maximal order $\mathcal O_\mu$, where $\mathcal O_\mu$ ranges over the right orders of a set of representatives for left ideal classes of a maximal order $\mathcal O$ in a quaternion algebra of discriminant $p$. In 2019, H.~Boylan, N.~Skoruppa and the second \cite{BSZ19} recovered the formula obtained by Gross, in the more general case of squarefree $N$, using the theory of Jacobi forms. More precisely, they showed that the Jacobi Eisenstein series whose Fourier coefficients are modified Hurwitz class number $H^{(N)}(4n-r^2)$
agrees with a weighted sum of theta series corresponding to a set of representatives for the conjugacy classes of $\mathcal O$. Recently, Y.~Li, N.~Skoruppa and the second author \cite{LSZ22} extended it to more general cases. More specifically, for all Eichler orders with a same squarefree level in a definite quaternion algebra over the field of rational numbers, they proved that a weighted sum of Jacobi theta series associated with these orders is a Jacobi Eisenstein series which has Fourier coefficients   $H^{(N_1,N_2)}(4n-r^2)$. 


Consider the bijections between Eichler orders and ternary quadratic forms and associate  the results of Y.~Li et al.\cite{LSZ22} and their modified Hurwitz class number $H^{(N_1,N_2)}(D)$, we derive  the explicit formula of for the  average number of representations over the genus of definite positive ternary quadratic form of level $4N$.

\begin{theorem}\label{theo:rep}
For any squarefree poisitive integer $N$ and any divisors $\Nodd$ and $\Neven$ of $N$ with an odd respectively even number of prime factors, and for any nonnegative integer $n$, one has
\begin{equation*}
\sum\limits_{f\in G_{4N,N^2/N_r,\Nodd}}\frac{R_f(n)}{|\Aut(f)|}=2^{-s-1}H^{(\Nodd,N/\Nodd)}(4N_rn),
\end{equation*}
\begin{equation*}
\sum\limits_{f\in G_{4N,4N^2/N_r,\Nodd}}\frac{R_f(n)}{|\Aut(f)|}=2^{-s-2}H^{(\Nodd,2N/\Nodd)}(4N_rn),
\end{equation*}
\begin{equation*}
\sum\limits_{f\in G_{4N,4N^2/N_r,2\Neven}}\frac{R_f(n)}{|\Aut(f)|}=2^{-s-2}H^{(2\Neven,N/\Neven)}(4N_rn),
\end{equation*}
and
\begin{equation*}
\sum\limits_{f\in G_{4N,16N^2/N_r,\Nodd}}\frac{R_f(n)}{|\Aut(f)|}=2^{-s-1}H^{(\Nodd,N/\Nodd)}(N_rn).
\end{equation*}
Here as above $R_f(n)$ denotes the number of representations of $n$ by the form $f$, we use
$\Aut(f)$ for the number of  automorphs of $f$ (see before Prop.~\ref{m1}), and $H^{(N_1,N_2)}(D)$ is the modified Hurwitz class number defined by~\eqref{Hurwitz-class-number}in Section~\ref{sec:quaternion-algebra}.  The sums are over a complete set of equivalent classes (see the beginning of Sect.~\ref{ssec:Lemman's-bijection}) in the given genus classes, respectively.
\end{theorem}

\begin{remark}\label{re1}
	The conditions hold for $\Neven=1$ and $N_r=1$. When $N=1$, all positive definite ternary quadratic forms of level $4$ are only in one genus which has one class, and the sum of three squares is their representative element
	\begin{equation*}
		r_3(n)=12H^{(2,1)}(4n).
	\end{equation*}
\end{remark}

If the class number of the genus of a positive definite ternary quadratic forms $f$ of level $4N$ is one,  we can give the explicit formula of $R_f(n)$.  
In the table 2 in section 8, we list 73 formulas of $R_f(n)$.

Li et al. \cite{LSZ22} also gave a simple formula for the type number of Eichler orders with squarefree level by modified Hurwitz class number. Based on the type number formula, we can give the formula of the class number  of ternary quadratic forms of level $4N$.
\begin{theorem}\label{theo:class-num}
Let $|C(4N)|$ denote the number of classes of primitive positive definite ternary quadratic forms of level $4N$. Then we have
\begin{align*}
|C(4N)|=2^{s}\left(\frac{N}{6}+\frac{5}{4}-\frac{1}{4}\left(\frac{-4}{N}\right)-\frac{1}{6}\left(\frac{-3}{N}\right)+\frac{1}{2}\left(1-\left(\frac{N}{3}\right)^2\right)+\frac{1}{4}\sum_{\substack{{d\mid N}\\d\neq1}}\left(3H(4d)+H(8d)\right)\right).
\end{align*}
\end{theorem}

In the Appendix B, we give a list of $|C(4N)|$, the number of classes of primitive positive definite ternary quadratic forms of level $4N$  for $N<1000$.


In Section 7, we give three applications of theorem 2. 
Firstly, we  give a new explicit basis of the Eisenstein space $\mathscr{E}(4N, \frac{3}{2}, \chi_l)$, where $\mathscr{E}(4N, \frac{3}{2}, \chi_l)$ is the orthogonal complement of the subspace of cusp forms in the complex linear space of modular forms of weight $3/2$, level $4N$ and  character $\chi_l$ with respect to Petersson inner product. 
Secondly, we give a new proof of Berkovich and Jagy’s genus identity \cite{BJa12} 
\begin{equation}
r_3(p^2n)-pr_3(n)=48\sum\limits_{f\in TG_{1,p}}\frac{R_{f}(n)}{|\Aut(f)|}-96\sum\limits_{f\in TG_{2,p}}\frac{R_{f}(n)}{|\Aut(f)|}.
\end{equation}
Finally, we give a new proof of Du’s interesting equality \cite{Du16} for the case positive ternary quadratic forms,
\begin{equation}
-\frac{2}{q-1}r_{Dp,N}(m)+\frac{q+1}{q-1}r_{Dp,Nq}(m)=-\frac{2}{p-1}r_{Dq,N}(m)+\frac{p+1}{p-1}r_{Dq,Np}(m).
\end{equation}

 In Section 2 we recall some tools on ternary quadratic form. In Section 3 we introduce the basic konwledge on definite quaternion algebras. In Section 4, we give some relations between ternary quadratic forms and Eichler orders. In Section 5, we will prove Theorem 1.1. In Section 6, we will discuss the connection among $3$ bijections between Eichler orders and ternary quadratic forms and give proof of Theorem 1.2 and Theorem 1.3. In Section 7, we give some applications. In Section 8, we give some representation numbers of ternary quadratic forms in the tables, and some calculations are aided by Sagemath.

\section{Ternary quadratic forms}
\subsection{Level and discriminant}
 Let $f$ be a ternary quadratic form with integer coefficients, given by the equation
\begin{equation*}
f(x,y,z)=ax^2+by^2+cz^2+ryz+sxz+txy. 
\end{equation*}
Define the divisor of $f$ to be the positive integer
\begin{equation*}
m=m_f=\text{gcd}(M_{11},M_{22},M_{33},2M_{23},2M_{13},2M_{12}),
\end{equation*}
where
\begin{center}
$M_{11}=4bc-r^2,M_{23}=st-2ar=M_{32},$\\
$M_{22}=4ac-s^2,M_{13}=rt-2bs=M_{31},$\\
$M_{33}=4ab-t^2,M_{12}=rs-2ct=M_{21}.$
\end{center}
The level of $f$ is   equivalent to
$
4d_f/m_f $  \cite[pp.401-402]{Leh92}.

 There is a connection between level and discriminant, as follows.
\begin{theorem}\label{Leh1}\cite[Theorem 2]{Leh92}
Let $f$ be a primitive positive definite ternary quadratic form of level $N'$ and discriminant $d$. 
\begin{equation}
N'=2^{n_0}p_1^{n_1}...p_k^{n_k}\label{N}
\end{equation}
is the prime factorization of $N'$. Then $n_0\geq2$ and $d$ is of the form 
\begin{equation}
d=2^{d_0}p_1^{d_1}...p_k^{d_k}\label{d}
\end{equation}
with the following restrictions on exponents:
\par (1) $d_0=n_0-2,d_0=2n_0,$ or $n_0\leq d_0\leq 2n_0-2$, and
\par (2) for $1\leq i\leq k, n_i\leq d_i\leq 2n_i$\\
Furthermore, if $n_i$ is even for $0\leq i\leq k$, then either $n_0\leq d_0\leq 2n_0-2$, or $d_i$ is odd for some $1\leq i\leq k$.
\end{theorem}

  The above theorem establishes that discriminants of primitive positive definite ternary quadratic forms of level $4N$ belong to one of three categories: $N^2/N_r$, $4N^2/N_r$ and $16N^2/N_r$. In terms of the equivalent form of the level, it is ease to show that the positive definite ternary quadratic forms of level $4N$ and discriminant $N^2/N_r$ or $4N^2/N_r$ or $16N^2/N_r$  must be primitive.

\subsection{Lehman's bijection}\label{ssec:Lemman's-bijection}
Ternary quadratic forms $f$ and $g$ are said to be equivalent, denoted as $f\sim g$, if there exists a matrix $M\in GL_3(\mathbb{Z})$ (i.e., with integer entries and $\det(M)=\pm1$) such that $M_{f}=MM_{g}M^{t}$. Equivalent forms are said to belong to the same class. If $f \sim  g$, then $d_f=d_g$ and $N_f=N_g$\cite[pp.401-402]{Leh92}. Let us  denote by $C(N,d)$ the set of all classes of primitive positive definite ternary quadratic forms of level $N$ and discriminant $d$.
\begin{theorem}\label{Leh2}\cite[Theorem 4, Theorem 5]{Leh92}
Let $N'$ and $d$ be given by equations (\ref{N}) and (\ref{d}). Suppose that $p^g\parallel N'$ and $p^h\parallel d$ for some odd prime p. Write $d$ as $p^hd'$. Then there is a one-to-one correspondence between $C(N',p^hd')$ and $C(N',p^{3g-h}d')$. If $p=2$, then there is a one-to-one correspondence between $C(N',2^hd')$ and $C(N',2^{3g-h-2}d')$.
\end{theorem}
Before we describe this correspondence, we need the following lemma.
\begin{lemma}\label{Leh3}\cite[Lemma 2]{Leh92}
Let $f$ be a primitive positive definite ternary quadratic forms of level $N'$ and divisor $m$. Suppose that $p^i\parallel N'$ and $p^j\parallel m$ for some odd prime $p$ and positive integer $i$. Then $f$ is equivalent to a form $(a,b,c,r,s,t)$ with $p^i\parallel a$, $p^i\mid s$ and $t$, $p^j\mid b$ and $r$, and $p\nmid c$. If $0<j<i$, then we can assume that $p^j\parallel b$.
\end{lemma}

\begin{remark}\label{rem1}
Let $N$ be a product of distinct odd primes. For $f\in C(4N,N)$,  since $N=4d/m$, then $m=1$.  Hence, according to the Lemma \ref{Leh3} we can assume that $f$ is equivalent to a form $(Na,b,c,r,Ns,Nt)$  with $(N,ac)=1$. 
\end{remark}

Let $f \in C(N',p^hd')$, $p^g\parallel N'$ and $p \nmid d'$. In terms of Lemma \ref{Leh3}, we can assume that 
\begin{equation*}
f=(p^ga,p^{h-g}b,c,p^{h-g}r,p^gs,p^gt),
\end{equation*}
with $a, b, c, r, s$ and $t$ integers, $p\nmid ac$. Define the map
\begin{align*}
\phi_p :C(N',p^hd') & \to C(N',p^{3g-h}d')\\
f& \mapsto (a,p^{2g-h}b,p^gc,p^{g}r,p^gs,p^{2g-h}t)\\
M_f & \mapsto PM_{f}P,
\end{align*}
 where 
 \begin{equation*}
 P=
\begin{pmatrix}
p^{-g/2} & 0& 0\\
0 & p^{(3g-2h)/2}& 0\\
0 & 0& p^{g/2}
\end{pmatrix}.
\end{equation*}

Lehman proved that $\phi_p$ is a one-to-one correspondence.  
In particular, for $p|N$,  maps  $ \phi_{p}  :C(4N, 4^u N^2)  \to C(4N,4^u N^2/p), u\in\{0,1,2\} $ are also  one-to-one correspondences. Let $N_r= p_1p_2\cdots p_r$ be a divisor of $N$. 
   Composite maps $\phi_{p_1}\circ...\circ\phi_{p_r} :C(4N,4^uN^2)  \to C(4N,4^uN^2/N_r)$ are also one-to-one correspondences.

For $p=2$, in terms of Lemma \ref{Leh3}, the map $\phi_{2} :C(4N,N^2/N_r) \to C(4N,16N^2/N_r).$

Two integral quadratic forms are said to be semi-equivalent if they are equivalent over the $p$-adic integers for all primes $p$, and are equivalent over the real numbers. Semi-equivalent forms are said to be in the same genus of forms. Equivalent forms are semi-equivalent, so we may speak of a class of forms as belonging to a genus \cite[pp.409-410]{Leh92}. We denote $\Aut(f)$ the finite group of integral automorphs of $f$ (that is, $\Aut(f)$ is the set of all $M\in GL_3(\mathbb{Z})$ such that $M_f=M^{t}M_fM$).

It is the same as Theorem \ref{Leh2} to prove the following proposition.
\begin{prop}\label{m1}
	Let $p$ be a prime, and $f\in C(N',p^hd')$ where $p\nmid d'$, we have
	\par (1) $|\Aut(f)| =|\Aut(\phi_p(f))|$,
	\par (2) if $f$ and $g$ are semi-equivalent, then $\phi_p(f)$ and $\phi_p(g)$ are semi-equivalent.
\end{prop}

\subsection{Watson transformation}

Recalling Watson transformation \cite{Wat62}, let $\Lambda_f$ be a 3-dimensional lattice, where $f$ is $\Lambda_f$ correspondence integral ternary quadratic form. Define $\Lambda_m(f)$ to be the set of all $\boldsymbol{x}$ in $\Lambda_f$ with
\begin{equation*}
f(\boldsymbol{x}+\boldsymbol{z})\equiv f(\boldsymbol{z})\pmod m, \forall \boldsymbol{z}\in\Lambda_f.
\end{equation*}
$\Lambda_m(f)$ is a 3-dimensional lattice \cite[p.578]{Wat62}. Hence we can choose $M$ so that $\boldsymbol{x}$ is in $\Lambda_m(f)$ if and only if $\boldsymbol{x}=M\boldsymbol{y},  y \in \Lambda_f$. Let $g(\boldsymbol{y})=m^{-1}f(M\boldsymbol{y})$, that is
\begin{equation*}
\frac{1}{2}\boldsymbol{y}^{t}M_g\boldsymbol{y}=\frac{1}{2}\cdot m^{-1}(M\boldsymbol{y})^tM_f(M\boldsymbol{y}),\boldsymbol{y}\in\Lambda_f.
\end{equation*}

\begin{theorem}\cite[p.579]{Wat62}\label{Watsontrans}
Watson transformation $\lambda_m$ is a well-define mapping, that is if $f\sim g$, then $\lambda_m(f)\sim \lambda_m(g)$, and if $f$  and $g$ are semi-equivalent, then $\lambda_m(f)$  and $\lambda_m(g)$ are semi-equivalent. Furthermore, the class number does not increase under the Watson transformation, that is,  $\lambda_m(f)$ can range over the whole of the genus of $\lambda_m(f)$.
\end{theorem}

Watson transformation is usual a surjection between two genera. Now we will show that Watson transformation $\lambda_4$ is a bijection for some given level and discriminants.

\begin{prop}\label{OQTC}
Watson transformation $\lambda_4$ is a bijection between $C(8N,64N^2/N_r)$(resp. $C(4N,16N^2/N_r)$) and  $C(4N,4N^2/N_r)$ (resp. $C(4N,N^2/N_r)$), and we have
\begin{equation*}
|\Aut(f)| =|\Aut(\lambda_4(f))|.
\end{equation*}
\end{prop}
\begin{proof}
We will show that all positive definite ternary quadratic forms of level $8N$ and discriminants $64N^2/N_r$ are primitive. If $f$ is a positive definite ternary quadratic forms of level $8N$ and discriminants $64N^2/N_r$ but not primitive, then $f/2$ or $f/4$ is a primitive positive definite ternary quadratic forms of level $4N$ or $2N$ and discriminants $8N^2/N_r$ or $N^2/N_r$. By Theorem \ref{Leh1}, this is a contradiction. Assume that
\begin{equation*}
f(x,y,z)=8ax^2+8by^2+cz^2+8ryz+8sxz+8txy
\end{equation*}
where $2\nmid ac$ \cite[Theorem 5]{Leh92}. Let $m=4$, then it is not hard to check
\begin{equation*}
M=
\begin{pmatrix}
1 & 0& 0\\
0 & 1& 0\\
0 & 0& 2
\end{pmatrix}
\end{equation*}
and
\begin{equation*}
g(x,y,z)=2ax^2+2by^2+cz^2+4ryz+4sxz+2txy.
\end{equation*}
We have $d_g=4N^2$. Since
\begin{equation*}
8NM_g^{-1}=(2M^{-1})(8NM_f^{-1})(2M^{-1})
\end{equation*}
is even, by Theorem \ref{Leh1} we have $N_g=4N$. Hence $g\in C(4N,4N^2)$. Let
\begin{equation*}
P=\frac{1}{2}M=
\begin{pmatrix}
\frac{1}{2} & 0& 0\\
0 & \frac{1}{2}& 0\\
0 & 0& 1
\end{pmatrix}.
\end{equation*}
Now we will prove this map is a bijection. Since this mapping is a surjection, we need to show this mapping is injective. Suppose that $g,g'\in C(4N,4N^2)$ are equivalent and $g$ (resp. $g'$) is equivalent to a form $(2a,b,c,4r,4s,2t)$ (resp. $(2a',b',c',4r',4s',2t')$) with $2\nmid ac$ (resp. $2\nmid a'c'$). There is an integral matrix $U=\{u_{ij}\}_{3\times3}\in GL_{3}(\mathbb{Z})$ such that $U^{t}M_{g}U=M_{g'}$. Then we have $4a'\equiv 2cu_{31}^2\pmod4$ and $4b'\equiv 2cu_{32}^2\pmod4$. Hence $u_{31}$ and $u_{32}$ are even. Similarly $4r'\equiv 2u_{13}u_{22}t+2u_{12}u_{23}t+2u_{32}u_{33}c\pmod4$ and $4s'\equiv 2u_{13}u_{21}t+2u_{11}u_{23}t+2u_{31}u_{33}c\pmod4$. Since $d_g=4N^2=8abc+32rst-32ar^2-16bs^2-4ct^2$, we have $2\nmid t$. Then $2\mid u_{13}u_{22}+u_{12}u_{23}$ and $2\mid u_{13}u_{21}+u_{11}u_{23}$. Let $U^{-1}=\pm\{A_{ji}\}_{3\times3}$, where $\pm$ depends on $\mbox{sgn}(\det(U))$, then we have
\begin{equation*}
M_g=(U^{-1})^tM_{g'}U^{-1}
\end{equation*}
and $A_{31}=u_{12}u_{23}-u_{13}u_{22}, A_{32}=u_{11}u_{23}-u_{21}u_{13}$. Hence $A_{31}$ and $A_{32}$ are even, and 
\begin{equation*}
PU^{-1}P^{-1}=\pm
\begin{pmatrix}
A_{11} & A_{21}& \frac{1}{2}A_{31}\\
A_{12} & A_{22}& \frac{1}{2}A_{32}\\
2A_{13} & 2A_{23}& A_{33}
\end{pmatrix}
\end{equation*}
is integral. It follows that
\begin{align*}
M_{f}&=P^{-1}M_{g}P^{-1}\\
&=P^{-1}(U^{-1})^tM_{g'}U^{-1}P^{-1}\\
&=P^{-1}(U^{-1})^tPP^{-1}M_{g'}P^{-1}PU^{-1}P^{-1}\\
&=(PU^{-1}P^{-1})^tP^{-1}M_{g'}P^{-1}(PU^{-1}P^{-1})\\
&=(PU^{-1}P^{-1})^tM_{f'}(PU^{-1}P^{-1}).
\end{align*}
Hence $f\sim f'$. It follows that this map is injective. For $f\in C(4N,16N^2/N_r)$, it is not hard to check that $\lambda_4$ is the same as $\phi_2$. It is easily to prove that
\begin{equation*}
|\Aut(f)| =|\Aut(\lambda_4(f))|.
\end{equation*}
\end{proof}
\subsection{Isotropic}
Let $f$ be a ternary quadratic forms over $\mathbb{Q}$. It can be equivalent over $\mathbb{Q}$ to a diagonal form
\begin{equation*}
ax^2+by^2+cz^2.
\end{equation*}
The Hasse invariant for $f$ at $p$ is defined to be
\begin{equation*}
S_p(f)=(a,-1)_p(b,-1)_p(c,-1)_p(a,b)_p(b,c)_p(c,a)_p,
\end{equation*}
which depends only on the equivalence class of $f$, and $(a,b)_p$ is the Hilbert Norm Residue Symbol. We say $f_p$ is isotropic if $f_p(x,y,z)=0$ for some non-zero elements $x,y,z\in\mathbb{Q}_p$ , and it is anisotropic otherwise, where $f_p$ means the localization of $f$ at $p$. The modified Hasse symbol $S^*_p$ indicates whether $f_p$ is isotropic or not, that is
\begin{equation*}
S^*_p(f)=(-1)^{\delta_{p,2}}S_p(f)=
\begin{cases}
1 & \text{if $f_p$ is isotropic},\\
-1 & \text{if $f_p$ is anisotropic},
\end{cases}
\end{equation*}
where $\delta_{m,n}=1$ if $m=n$ and otherwise $0$. Recall that Hilbert’s reciprocity law says $(a,b)_\infty\prod_p(a,b)_p=1$, which futher leads to $S_\infty(f)\prod_pS_p(f)=1$. If $f$ is a primitive positive definite ternary quadratic forms, it is not hard to check that $S_\infty(f)=1$. Hence $f$ is anisotropic only at odd number of primes.

It is not hard to check the following proposition.
\begin{prop}\label{ISO}
(1)Let $p,q$ be primes, and $f\in C(N',p^hd')$ where $p\nmid d'$, then $f$ is isotropic at $q$ if and only if  $\phi_p(f)$ is isotropic at $q$.\\
(2)Let $f\in C(8N,64N^2/N_r)$ or $f\in C(4N,16N^2/N_r)$, then $f$ is isotropic at $q$ if and only if  $\lambda_4(f)$ is isotropic at $q$.
\end{prop}
\begin{proof}
Let $f \in C(N',p^hd')$, $p^g\parallel N'$ and $p \nmid d'$. In terms of Lemma \ref{Leh3}, we can assume that 
\begin{center}
$f=(p^ga,p^{h-g}b,c,p^{h-g}r,p^gs,p^gt)$,
\end{center}
with $a, b, c, r, s$ and $t$ integers, $p\nmid ac$. The same process that took some $f$ and constructed automorphs or equivalences involving $\phi$ can be readily extended to the $q$-adic integers. If $f_q(x,y,z)=0$ for some non-zero elements $x,y,z\in\mathbb{Q}_q$, that is
\begin{equation*}
X^tM_fX=0, X=
\begin{pmatrix}
x\\
y\\
z
\end{pmatrix}
,
\end{equation*}
then
\begin{equation*}
p^g(P^{-1}X)^tPM_fP(P^{-1}X)=(p^{g/2}P^{-1}X)^tM_{\phi_p(f)}(p^{g/2}P^{-1}X)=0,
\end{equation*}
where
\begin{equation*}
 P=
\begin{pmatrix}
p^{-g/2} & 0& 0\\
0 & p^{(3g-2h)/2}& 0\\
0 & 0& p^{g/2}
\end{pmatrix}.
\end{equation*}
For $\phi_2$ and $\lambda_4$, it is similar to the above.
\end{proof}

\section{Quaternion algebras}\label{sec:quaternion-algebra}

In this section we will begin with the basic knowledge about quaternion algebras. For more detailed definition of quaternion algebras one can see \cite{Voi21}.

For a given filed $F$ of characteristic $0$ and elements $a,b\in F^{\times}$, we use $Q=\left(\frac{a,b}{F}\right)$ for the $F$-algebra which possesses a basis $1,i,j,k=ij=-ji$ with $i^2=a, j^2=b$. In this paper $F$ only represents $\mathbb{Q}$ or the fields $\mathbb{Q}_p$ of the $p$-adic numbers where $p$ is a prime. A quaternion algebra $Q$ is a central simple algebra. For the field of $p$-adic numbers $\mathbb{Q}_p$ or of real numbers $\mathbb{Q}_{\infty}$, there are up to isomorphism only two such equivalence classes, the algebras $M_2(\mathbb{Q}_p)$ of $2\times2$ matrices over $\mathbb{Q}_p$, or skew-fields. We use the usual $p$-adic Hilbert symbol $(a,b)_p$. The quaternion algebra $Q_p=\left(\frac{a,b}{\mathbb{Q}_p}\right)\simeq M_2(\mathbb{Q}_p)$ if and only if $(a,b)_p=1$. In the case of a quaternion algebra over $\mathbb{Q}$ the algebras  $Q=\left(\frac{a,b}{\mathbb{Q}}\right)$ and $Q'=\left(\frac{a',b'}{\mathbb{Q}}\right)$ are isomorphic if and only if $(a,b)_p=(a',b')_p$ for all $p$ and $\infty$. A quaternion algebra $Q$ over $\mathbb{Q}$ ramifies at $p$ if $\mathbb{Q}_p\otimes Q$ is a skew-field, and splitts at $p$ if $\mathbb{Q}_p\otimes Q\simeq M_2(\mathbb{Q}_p)$. If it ramifies at $\infty$ it is called definite. A quaternion algebra $Q$ ramifies only at finitely many primes and this number is always even. Hence given a set $S$ of an even number of primes, we can find $a,b$ in $\mathbb{Q}^{\times}$ such that $(a,b)_p=-1$ exactly for $p$ in $S$, and then $Q=\left(\frac{a,b}{F}\right)$ is a (up to isomorphism) quaternion algebra which ramifies exactly at the primes in $S$.
\par For any element $\alpha=t+xi+yj+zk$ of quaternion algebra $Q$, we define an involution $\overline{\alpha}=t-xi-yj-zk$, which satisfies $\overline{\bar{\alpha}}=\alpha$ and $\overline{\alpha\beta}=\bar{\beta}\bar{\alpha}$ for $\beta\in Q$. The reduced trace on $Q$ is $\tr(\alpha)=\alpha+\bar{\alpha}$, and similarly the reduced norm is $\n(\alpha)=\alpha\cdot \bar{\alpha}$. If $Q\simeq M_2(F)$, and
 \begin{equation*}
 A=
\begin{pmatrix}
 a& b\\
 c& d
\end{pmatrix}\in Q,
\end{equation*}
then
 \begin{equation*}
 \bar{A}=
\begin{pmatrix}
 d& -b\\
 -c& a
\end{pmatrix}.
\end{equation*}
Hence $\tr(A)=a+d, \n(A)=ad-bc=\det(A)$. We can see that $a\in Q$ is a root of the polynomial
\begin{equation*}
x^2-tr(a)x+n(a)=0
\end{equation*}
which is the reduced characteristic polynomial of $a$.

An ideal of quaternion algebra $Q$ is a full $\mathbb{Z}$-lattice of $Q$, that is, a finitely generated $\mathbb{Z}$-submodule of $Q$ which contains a basis of $Q$ over $\mathbb{Q}$. An order $\mathcal{O}$ in a quaternion algebra $Q$ is an ideal which is also a ring containing $\mathbb{Z}$. For the rest of this paper $\mathcal{O}$ will always denote a quaternion order. If $x\in\mathcal{O}$, then $tr(x),n(x)\in \mathbb{Z}$. An order $\mathcal{O}$  is maximal if it is not properly contained in anorher order. An Eichler order $\mathcal{O}$ is the intersection of two (not necessarily distinct) maximal orders. For a quaternion algebra over $\mathbb{Q}_p$, Eichler orders can be described more explicitly. A quaternion division algebra $Q_p$ over $\mathbb{Q}_p$ contains only one maximal order, which is $K+Kj$ with $K=\mathbb{Z}_p+i\mathbb{Z}_p$ for odd $p$ and $K=\mathbb{Z}_2+\frac{i+1}{2}\mathbb{Z}_2$ for $p=2$, where $i^2=u,j^2=p$ for some $\left(\frac{u}{p}\right)=-1$. If $Q_p$ is isomorphic to $M_2(\mathbb{Q}_p)$, there is a characterization of Eichler orders.
\begin{theorem}\label{Eo}\cite[p.372]{Voi21}
For an order $\mathcal{O}\subset Q_p\simeq M_2(\mathbb{Q}_p)$, the following are equivalent:
\par (1) $\mathcal{O}$ is Eichler;
\par (2) $\mathcal{O}\simeq 
\begin{pmatrix}
\mathbb{Z}_p & \mathbb{Z}_p\\
p^n\mathbb{Z}_p & \mathbb{Z}_p\\
\end{pmatrix};$
\par (3) $\mathcal{O}$ is the intersection of a uniquely determined pair of maximal orders (not necessarily distinct).
\end{theorem}
The power $p^n$ is called the level of Eichler order $\mathcal{O}$. In particular, if $n=0$, then $\mathcal{O}$ is a maximal order.

The reduced discriminant of $Q$,$\disc(Q)$, is the product of primes at which $Q$ ramifies. Let $I$ be the ideal generated by all $\det(\tr(x_i\bar{x_j}))$, where $x_1,...,x_4\in \mathcal{O}$. It is easy to prove that $I$ is the square of an ideal. The reduced discriminant, $\discrd(\mathcal{O})$, of $\mathcal{O}$ is defined to be the square root of the $I$. The most important invariant of a quaternion order is the reduced discriminant. Eichler orders with squarefree level are also called Hereditary orders. Hereditary orders can be detected over global rings in terms of discriminants, as follows.
\begin{theorem}\label{Eod}\cite[p.241,p371]{Voi21}
Let $\mathcal{O}\subset Q$ be an order, and $\disc(Q)=N$.
\par (1) $\mathcal{O}$ is maximal if and only if $\discrd(\mathcal{O})=N$.
\par (2) $\mathcal{O}$ is hereditary  if and only if $\discrd(\mathcal{O})$ is squarefree, where $\discrd(\mathcal{O})=NF,(N,F)=1$ and $F$ is squarefree.
\end{theorem}
We say $\mathcal{O},\mathcal{O}'$ are of the same type if there exists $\alpha\in Q^{\times}$ such that $\mathcal{O}'=\alpha^{-1}\mathcal{O}\alpha$. Two orders $\mathcal{O},\mathcal{O}'$ are of the same type if and only if they are isomorphic as $\mathbb{Z}$-algebras. Two orders $\mathcal{O},\mathcal{O}'$ are locally of the same type or locally isomorphic if $\mathcal{O}_p,\mathcal{O}_p'$ are of the same type for all primes $p$. The genus of $\mathcal{O}$ is the set of orders in $Q$ locally isomorphic to $\mathcal{O}$. The type set of $\mathcal{O}$ is the set of isomorphism classes of orders in genus of $\mathcal{O}$. The number of the type set of  $\mathcal{O}$ is usually called the type number of $\mathcal{O}$. For all Eichler orders with  same squarefree level in a definite quaternion algebra over the field of rational numbers, Li, Skoruppa and Zhou proved that a weighted sum of Jacobi theta series associated to these orders is a Jacobi Eisenstein series. For any pair of relatively prime positive squarefree integers $(N_1, N_2)$ and any negative discriminant $-D$, Hurwitz class numbers $H(D)$ can be modified as follows:
\begin{equation}\label{Hurwitz-class-number}
H^{(N_1,N_2)}(D)=H(D/f^2_{N_1,N_2})\prod\limits_{p\mid N_1}\left( 1-\left(\frac{-D/f^2_{N_1,N_2}}{p}\right)\right)\prod\limits_{p\mid N_2}\frac{2pf_p-p-1-\left(\frac{-D/f^2_{N_1,N_2}}{p}\right)\left(2f_p-p-1\right)}{p-1},
\end{equation}
where $f_{N_1,N_2}$ is the largest positive integer containing only prime factors of $N_1N_2$ whose square divides $D$ such that $-D/f^2_{N_1,N_2}$ is still a negative discriminant. The products run through all primes $p$ dividing $N_1$ and $N_2$, respectively. We use $f_p$ for the exact $p$-power dividing $f_{N_1,N_2}$. In particular, when $f_p=1$, the above fraction containing $f_p$ becomes $1+\left(\frac{-D/f^2_{N_1,N_2}}{p}\right)$. All $(\frac{\cdot}{p})$ and $(\frac{\cdot}{l})$ are the Kronecker symbols, and for integer $m$, Kronecker symbol
\begin{equation*}
\left(\frac{m}{2}\right)=
\begin{cases}
1 & \text{if}\quad m\equiv\pm1 \pmod8,\\
-1 & \text{if}\quad m\equiv\pm3 \pmod8,\\
0 & \text{otherwise}.
\end{cases}
\end{equation*}
Set
\begin{equation*}
H^{(N_1,N_2)}(0)=-\frac{1}{12}\prod\limits_{p\mid N_1}(1-p)\prod\limits_{p\mid N_2}(1+p),
\end{equation*}
and $H^{(N_1,N_2)}(D)=0$ for every positive integer $D\equiv1,2\pmod4$. 
\begin{theorem}\label{LSZ}\cite{LSZ22}
Let $N$ and $F$ be two squarefree positive integers which are coprime, where N has an odd number of prime factors. Use $T_{N,F}$ for the type number of Eichler orders of level $F$ in $Q_N$, where $Q_N$ ramifies only at the primes which divides $N$. Choose a complete set of representatives $\mathcal{O}_{\mu}(\mu=1,2,...,T_{N,F})$  for these types of Eichler orders. 
Let
\begin{equation*}
\theta_{\mathcal{O_\mu}}=\sum_{\substack{n,r\in\mathbb{Z}\\4n-r^2\geq0}}\rho_{\mathcal{O_\mu}}(n,r)q^n\zeta^r
\end{equation*}
where $\rho_{\mathcal{O_\mu}}(n,r)$ is the number of zeros of $x^2-rx+n$ in $\mathcal{O}_{\mu}$. Then we have
\begin{equation}
\sum\limits_{\mu=1}^{T_{N,F}}\frac{\theta_{\mathcal{O_\mu}}}{\textnormal{card}(\Aut(\mathcal{O}_\mu))}=2^{-e(NF)}\sum_{\substack{n,r\in\mathbb{Z}\\4n-r^2\geq0}}H^{(N,F)}(4n-r^2)q^n\zeta^r,
\end{equation}
where $e(NF)$ is the number of prime factors of $NF$, and $\textnormal{card}(\Aut(\mathcal{O}_{\mu}))$ is the number of elements in the group of automorphisms of $\mathcal{O}_{\mu}$. 
\end{theorem}
Our motivation is the conection between Eichler orders and ternary quadratic forms and applying the above Theorem.

\section{Maps between Eichler orders and ternary quadratic forms}
In this section  we will discuss three maps between Eichler orders and ternary quadratic forms.
\subsection{Two trace zero lattices attached to Eichler orders}
Now we will discuss the norm forms of the Eichler orders $\mathcal{O}$, and its restriction to ternary sublattices $\mathcal{O}^0$ and $S^0$.
\begin{prop}\label{41}
Let 
\begin{center}
$\mathcal{O}=\mathbb{Z}+\mathbb{Z}\alpha_1+\mathbb{Z}\alpha_2+\mathbb{Z}\alpha_3\subset Q$
\end{center}
be an Eichler order, then there exist $\tr(\alpha_1')=\tr(\alpha_2')=0$, and $\tr(\alpha_3')=1$, such that
\begin{center}
$\mathcal{O}=\mathbb{Z}+\mathbb{Z}\alpha_1'+\mathbb{Z}\alpha_2'+\mathbb{Z}\alpha_3'$.
\end{center}
\end{prop}
\begin{proof}
Let $\tr(\alpha_1)=a, \tr(\alpha_2)=b, \tr(\alpha_3)=c$, where $a,b$ and $c$ are integers. If $a\equiv0\pmod2$, let $\alpha_1'=\alpha_1-a/2$, then
\begin{center}
$\mathcal{O}=\mathbb{Z}+\mathbb{Z}\alpha_1'+\mathbb{Z}\alpha_2+\mathbb{Z}\alpha_3$
\end{center}
and $\tr(\alpha_1')=0$. If $a\equiv0\pmod1$, let $\alpha_1'=\alpha_1-(a-1)/2$, then
\begin{center}
$\mathcal{O}=\mathbb{Z}+\mathbb{Z}\alpha_1'+\mathbb{Z}\alpha_2+\mathbb{Z}\alpha_3$
\end{center}
and $\tr(\alpha_1')=1$. Hence we can assume that $a=0$ or $a=1$. For $\alpha_2$ and $\alpha_3$, it is similar to the above.\\ 
If $a=b=c=1$, let $\alpha_1'=\alpha_1-\alpha_3$ and $\alpha_2'=\alpha_2-\alpha_3$, then
\begin{center}
$\mathcal{O}=\mathbb{Z}+\mathbb{Z}\alpha_1'+\mathbb{Z}\alpha_2'+\mathbb{Z}\alpha_3'$,
\end{center}
$\tr(\alpha_1')=\tr(\alpha_2')=0$ and $\tr(\alpha_3')=1$.\\
If $a+b+c=2$, without loss of generality we assume that $a=0, b=c=1$. Let $\alpha_2'=\alpha_2-\alpha_3$, then
\begin{center}
$\mathcal{O}=\mathbb{Z}+\mathbb{Z}\alpha_1+\mathbb{Z}\alpha_2'+\mathbb{Z}\alpha_3'$,
\end{center}
$\tr(\alpha_1)=\tr(\alpha_2')=0$ and $\tr(\alpha_3')=1$.\\
If $a+b+c=1$, without loss of generality we assume that $a=b=0, c=1$, and it is well-done.\\
If $a=b=c=0$, we will prove it is a contradiction. If $Q$ ramifies at $2$, then
\begin{equation*}
\mathcal{O}_2\cong\mathbb{Z}_2+\mathbb{Z}_2\frac{1+i}{2}+\mathbb{Z}_2j+\mathbb{Z}_2\frac{1+i}{2}j,
\end{equation*}
where $i^2=-3,j^2=2$. We have 
\begin{equation*}
\mathbb{Z}_2+\mathbb{Z}_2\frac{1+i}{2}+\mathbb{Z}_2j+\mathbb{Z}_2\frac{1+i}{2}j=\mathbb{Z}_2+\mathbb{Z}_2\alpha_1'+\mathbb{Z}_2\alpha_2'+\mathbb{Z}_2\alpha_3',
\end{equation*}
where $\tr(\alpha_1')=\tr(\alpha_2')=0$, and $\tr(\alpha_3')=1$. Since $\tr(\alpha_i+\alpha_j)=\tr(\alpha_i)+\tr(\alpha_j)$, we have
\begin{equation*}
\begin{pmatrix}
1 \\
\frac{1+i}{2}\\
j \\
\frac{1+i}{2}j
\end{pmatrix}
=
\begin{pmatrix}
1 & 0 & 0 & 0 \\
x & & & \\
0 & &U & \\
0 & & &
\end{pmatrix}
\begin{pmatrix}
1 \\
\alpha_1' \\
\alpha_2' \\
\alpha_3'
\end{pmatrix},
\end{equation*}
where all the entries in the matrix are integers and $\det(U)\in\mathbb{Z}_p^{\times}$. It follows that
\begin{equation*}
1=\tr(\frac{1+i}{2})=x\tr(1)=2x.
\end{equation*}
Hence $x=1/2\notin\mathbb{Z}_2$. This is a contradiction. If $Q$ splitts at $2$, then 
\begin{equation*}
\mathcal{O}_{2}\cong 
\begin{pmatrix}
\mathbb{Z}_{2} & \mathbb{Z}_{2}\\
2^n\mathbb{Z}_{2} & \mathbb{Z}_{2}\\
\end{pmatrix}.
\end{equation*}
The rest of proof is similar to the above.
\end{proof}
Unless otherwise stated, we assume that for an Eichler order $\mathcal{O}=\mathbb{Z}+\mathbb{Z}\alpha_1+\mathbb{Z}\alpha_2+\mathbb{Z}\alpha_3$  and $\tr(\alpha_1)=\tr(\alpha_2)=0,\tr(\alpha_3)=1$.
For an Eichler order $\mathcal{O}=\mathbb{Z}+\mathbb{Z}\alpha_1+\mathbb{Z}\alpha_2+\mathbb{Z}\alpha_3\in Q$, lattice $\mathcal{O}^{0}=\mathcal{O}\cap Q^0=\mathbb{Z}\alpha_1+\mathbb{Z}\alpha_2+\mathbb{Z}(2\alpha_3-1)$, where $Q^0=\{\alpha\in Q:\tr(\alpha)=0\}$, is an even integral positive definite lattice, when equipped with the bilinear form $(x,y)\mapsto \tr(x\overline{y})$. We have the following ternary quadratic form
\begin{align*}
f_{\mathcal{O}^{0}} & =\n(x\alpha_1+y\alpha_2+z(2\alpha_3-1))\\
& =\n(\alpha_1)x^2+\n(\alpha_2)y^2+(4\n(\alpha_3)-1)z^2+2\tr(\alpha_2\overline{\alpha_3})yz+2\tr(\alpha_1\overline{\alpha_3})xz+\tr(\alpha_1\overline{\alpha_2})xy,
\end{align*}
and the Gram matrix of $f_{\mathcal{O}^{0}}$ is
\begin{equation*}
M_{f_{\mathcal{O}^{0}}}=
\begin{pmatrix}
\tr(\alpha_1\overline{\alpha_1}) &\tr(\alpha_1\overline{\alpha_2}) & 2\tr(\alpha_1\overline{\alpha_3})\\
\tr(\alpha_2\overline{\alpha_1}) &\tr(\alpha_2\overline{\alpha_2}) & 2\tr(\alpha_2\overline{\alpha_3})\\
2\tr(\alpha_3\overline{\alpha_1}) &2\tr(\alpha_3\overline{\alpha_2}) & 4\tr(\alpha_3\overline{\alpha_3})-2
\end{pmatrix}.
\end{equation*}

For an Eichler order $\mathcal{O}=\mathbb{Z}+\mathbb{Z}\alpha_1+\mathbb{Z}\alpha_2+\mathbb{Z}\alpha_3\in Q$, define the lattice
\begin{equation*}
S=\mathbb{Z}+2\mathcal{O}=\mathbb{Z}+\mathbb{Z}2\alpha_1+\mathbb{Z}2\alpha_2+\mathbb{Z}2\alpha_3,
\end{equation*}
then lattice $S^{0}=(\mathbb{Z}+2\mathcal{O})\cap Q^0=\mathbb{Z}2\alpha_1+\mathbb{Z}2\alpha_2+\mathbb{Z}(2\alpha_3-1)$ is an even integral positive definite lattice, when equipped with the bilinear form $(x,y)\mapsto \tr(x\overline{y})$. We have the following ternary quadratic form
\begin{align*}
f_{S^{0}} & =\n(x(2\alpha_1)+y(2\alpha_2)+z(2\alpha_3-1))\\
& =4\n(\alpha_1)x^2+4\n(\alpha_2)y^2+(4\n(\alpha_3)-1)z^2+4\tr(\alpha_2\overline{\alpha_3})yz+4\tr(\alpha_1\overline{\alpha_3})xz+4\tr(\alpha_1\overline{\alpha_2})xy,
\end{align*}
and the Gram matrix of $f_{S^{0}}$ is
\begin{equation*}
M_{f_{S^{0}}}=
\begin{pmatrix}
4\tr(\alpha_1\overline{\alpha_1}) &4\tr(\alpha_1\overline{\alpha_2}) & 4\tr(\alpha_1\overline{\alpha_3})\\
4\tr(\alpha_2\overline{\alpha_1}) &4\tr(\alpha_2\overline{\alpha_2}) & 4\tr(\alpha_2\overline{\alpha_3})\\
4\tr(\alpha_3\overline{\alpha_1}) &4\tr(\alpha_3\overline{\alpha_2}) & 4\tr(\alpha_3\overline{\alpha_3})-2        
\end{pmatrix}.
\end{equation*}

\begin{prop}\label{OQE}
Let $\mathcal{O},\mathcal{O}'\subset Q$ be orders. Suppose $\mathcal{O},\mathcal{O}'$ are isomorphic (resp. locally isomorphic), then $f_{\mathcal{O}^{0}},f_{\mathcal{O}'^{0}}$ are equivalent (resp. semi-equivalent).
\end{prop}
\begin{proof}
Let $f$ (resp. $f'$) denote $f_{\mathcal{O}^{0}}$ (resp. $f_{\mathcal{O}'^{0}}$). It is similar to Proposition \ref{41} that for an order $\mathcal{O}$, there exist $\tr(\alpha_1')=\tr(\alpha_2')=\tr(\alpha_3')=0$ such that
\begin{center}
$\mathcal{O}=\mathbb{Z}\alpha_0+\mathbb{Z}\alpha_1+\mathbb{Z}\alpha_2+\mathbb{Z}\alpha_3$.
\end{center}
For $a=b=0, c=1$, we have
\begin{equation*}
\mathcal{O}=\mathbb{Z}+\mathbb{Z}\alpha_1+\mathbb{Z}\alpha_2+\mathbb{Z}\alpha_3=\mathbb{Z}\alpha_3+\mathbb{Z}\alpha_1+\mathbb{Z}\alpha_2+\mathbb{Z}(2\alpha_3-1)
\end{equation*}
and $\tr(\alpha_1')=\tr(\alpha_2')=\tr(2\alpha_3'-1)=0$. Let 
\begin{equation*}
\mathcal{O}=\mathbb{Z}\alpha_0+\mathbb{Z}\alpha_1+\mathbb{Z}\alpha_2+\mathbb{Z}\alpha_3,
\end{equation*}
\begin{equation*}
\mathcal{O}'=\mathbb{Z}\alpha_0'+\mathbb{Z}\alpha_1'+\mathbb{Z}\alpha_2'+\mathbb{Z}\alpha_3'.
\end{equation*}
Since $\mathcal{O}\simeq \mathcal{O}'$ if and only if there exists $\beta\in Q^{\times}$such that $\mathcal{O}'=\beta^{-1}\mathcal{O}\beta$, and $\tr(\beta^{-1}\alpha_i\beta)=\tr(\alpha_i),\tr(\alpha_i+\alpha_j)=\tr(\alpha_i)+\tr(\alpha_j)$, we have
\begin{equation*}
\begin{pmatrix}
\beta^{-1}\alpha_0\beta \\
\beta^{-1}\alpha_1\beta \\
\beta^{-1}\alpha_2\beta  \\
\beta^{-1}\alpha_3\beta
\end{pmatrix}
=
\begin{pmatrix}
\delta & * & * & * \\
0 & & & \\
0 & &U & \\
0 & & &
\end{pmatrix}
\begin{pmatrix}
\alpha_0' \\
\alpha_1' \\
\alpha_2' \\
\alpha_3'
\end{pmatrix},
\end{equation*}
where all the entries in the matrix are integers and $\delta\det(U)=\pm1$. Hence $U\in GL_3(\mathbb{Z})$. Since $\tr(\beta^{-1}\alpha_i\beta\overline{\beta^{-1}\alpha_j\beta})=\tr(\alpha_i\overline{\alpha_j})$, we have
\begin{equation*}
M_f=\begin{pmatrix}
\tr(\alpha_1\overline{\alpha_1}) &\tr(\alpha_1\overline{\alpha_2}) & \tr(\alpha_1\overline{\alpha_3})\\
\tr(\alpha_2\overline{\alpha_1}) &\tr(\alpha_2\overline{\alpha_2}) & \tr(\alpha_2\overline{\alpha_3})\\
\tr(\alpha_3\overline{\alpha_1}) &\tr(\alpha_3\overline{\alpha_2}) & \tr(\alpha_3\overline{\alpha_3})
\end{pmatrix}=
\begin{pmatrix}
\tr(\beta^{-1}\alpha_1\beta\overline{\beta^{-1}\alpha_1\beta}) &\tr(\beta^{-1}\alpha_1\beta\overline{\beta^{-1}\alpha_2\beta}) & \tr(\beta^{-1}\alpha_1\beta\overline{\beta^{-1}\alpha_3\beta})\\
\tr(\beta^{-1}\alpha_2\beta\overline{\beta^{-1}\alpha_1\beta}) &\tr(\beta^{-1}\alpha_2\beta\overline{\beta^{-1}\alpha_2\beta}) & \tr(\beta^{-1}\alpha_2\beta\overline{\beta^{-1}\alpha_3\beta})\\
\tr(\beta^{-1}\alpha_3\beta\overline{\beta^{-1}\alpha_1\beta}) &\tr(\beta^{-1}\alpha_3\beta\overline{\beta^{-1}\alpha_2\beta}) & \tr(\beta^{-1}\alpha_3\beta\overline{\beta^{-1}\alpha_3\beta})
\end{pmatrix}.
\end{equation*}
Let
\begin{equation*}
H=
\begin{pmatrix}
\alpha_1 \\
\alpha_2 \\
\alpha_3
\end{pmatrix}
\begin{pmatrix}
\overline{\alpha_1} & \overline{\alpha_2} & \overline{\alpha_3}
\end{pmatrix},
H'=
\begin{pmatrix}
\alpha_1' \\
\alpha_2' \\
\alpha_3'
\end{pmatrix}
\begin{pmatrix}
\overline{\alpha_1'} & \overline{\alpha_2'} & \overline{\alpha_3'}
\end{pmatrix}.
\end{equation*}
We have
\begin{equation*}
M_f=\begin{pmatrix}
\tr(\alpha_1\overline{\alpha_1}) &\tr(\alpha_1\overline{\alpha_2}) & \tr(\alpha_1\overline{\alpha_3})\\
\tr(\alpha_2\overline{\alpha_1}) &\tr(\alpha_2\overline{\alpha_2}) & \tr(\alpha_2\overline{\alpha_3})\\
\tr(\alpha_3\overline{\alpha_1}) &\tr(\alpha_3\overline{\alpha_2}) & \tr(\alpha_3\overline{\alpha_3})
\end{pmatrix}
=H+H^t, M_{f'}=H'+H'^{t}.
\end{equation*}
Since
\begin{equation*}
\begin{pmatrix}
\beta^{-1}\alpha_1\beta \\
\beta^{-1}\alpha_2\beta  \\
\beta^{-1}\alpha_3\beta
\end{pmatrix}
=U
\begin{pmatrix}
\alpha_1' \\
\alpha_2' \\
\alpha_3'
\end{pmatrix},
\begin{pmatrix}
\overline{\beta^{-1}\alpha_1\beta} &\overline{\beta^{-1}\alpha_2\beta} &\overline{\beta^{-1}\alpha_3\beta}
\end{pmatrix}
=
\begin{pmatrix}
\overline{\alpha_1'} & \overline{\alpha_2'} & \overline{\alpha_3'}
\end{pmatrix}U^t,
\end{equation*}
it follows that
\begin{equation*}
M_f=UH'U^t+(UH'U^t)^t=U(H'+H'^t)U^t=UM_{f'}U^t.
\end{equation*}
The rest of proof is similar to the above.
\end{proof}

For a quaternion algebra over $\mathbb{Q}_p$, Eichler orders can be described more explicitly. For an Eichler order $\mathcal{O}$, by Proposition \ref{OQE},  we can give the following characterization of the genus which $f_{\mathcal{O}^{0}}$ belongs to.
\begin{prop}\label{OQD} 
Let $N$ be a product of s distinct odd primes, and for an order $\mathcal{O}$ we denote $f_{\mathcal{O}^{0}}$ by $f$. 
\par (1)Let $\mathcal{O}\subset Q_{\Nodd}$ be an Eichler order of level $N/\Nodd$, then $d_f=N^2, N_f=4N$ and $f$ is anisotropic only in the $p$-adic field for $p\mid\Nodd$. The genus which $f$ belongs to is denoted by $G_{4N,N^2,\Nodd}$.
\par (2)Let $\mathcal{O}\subset Q_{\Nodd}$ be an Eichler order of level $2N/\Nodd$, then $d_f=4N^2, N_f=4N$ and $f$ is anisotropic only in the $p$-adic field for $p\mid\Nodd$. The genus which $f$ belongs to is denoted by $G_{4N,4N^2,\Nodd}$.
\par (3)Let $\mathcal{O}\subset Q_{2\Neven}$ be an Eichler order of level $N/\Neven$, then $d_f=4N^2, N_f=4N$ and $f$ is anisotropic only in the $p$-adic field for $p\mid2\Neven$. The genus which $f$ belongs to is denoted by $G_{4N,4N^2,2\Neven}$.
\end{prop}
\begin{proof}
Let $\mathcal{O}\subset Q_{\Nodd}$ be an Eichler order of level $N/\Nodd$. If $p\mid \Nodd$, then
\begin{equation*}
\mathcal{O}_p\cong\mathbb{Z}_p+i\mathbb{Z}_p+j\mathbb{Z}_p+ij\mathbb{Z}_p,
\end{equation*}
where $i^2=u,j^2=p$, and $(\frac{u}{p})=-1$. We have
\begin{equation*}
f\underset{p}{\sim}-ux^2-py^2+upz^2=f_p,
\end{equation*}
where $f\underset{p}{\sim}g$ means that $f$ and $g$ are equivlent over the $p$-adic integers $\mathbb{Z}_p$. It follows that $d_{f_p}=4u^2p^2,m_{f_p}=4up$. Hence $v_p(d_{f_p})=2,v_p(m_{f_p})=1$ and $S^*_p(f_p)=-1$. We have $f_p$ is anisotropic.

If $p\mid N/\Nodd$, then
\begin{equation*}
\mathcal{O}_{p}\cong 
\begin{pmatrix}
\mathbb{Z}_{p} & \mathbb{Z}_{p}\\
p\mathbb{Z}_{p} & \mathbb{Z}_{p}\\
\end{pmatrix}
\end{equation*}
and
\begin{equation*}
\mathcal{O}^0_{p}\cong
 \mathbb{Z}_{p}
\begin{pmatrix}
1 & 0\\
0 & -1\\
\end{pmatrix}
+ \mathbb{Z}_{p}
\begin{pmatrix}
0 & 1\\
0 & 0\\
\end{pmatrix}
+ \mathbb{Z}_{p}
\begin{pmatrix}
0 & 0\\
p & 0\\
\end{pmatrix},
\end{equation*}
It follows that
\begin{equation*}
f\underset{p}{\sim}-x^2-pyz=f_{p}.
\end{equation*}
We have $d_{f_{p}}=p^2,m_{f_{p}}=p$, then $v_{p}(d_{f_{p}})=2,v_{p}(m_{f_{p}})=1$ and $v_2(d_{f_2})=2,v_2(m_{f_2})=1$ and $f_p$ is isotropic.

If $p\nmid N$, then 
\begin{equation*}
\mathcal{O}_{p}\cong 
\begin{pmatrix}
\mathbb{Z}_{p} & \mathbb{Z}_{p}\\
\mathbb{Z}_{p} & \mathbb{Z}_{p}\\
\end{pmatrix}
\end{equation*}
and
\begin{equation*}
\mathcal{O}^0_{p}\cong
 \mathbb{Z}_{p}
\begin{pmatrix}
1 & 0\\
0 & -1\\
\end{pmatrix}
+ \mathbb{Z}_{p}
\begin{pmatrix}
0 & 1\\
0 & 0\\
\end{pmatrix}
+ \mathbb{Z}_{p}
\begin{pmatrix}
0 & 0\\
1 & 0\\
\end{pmatrix}.
\end{equation*}
It follows that
\begin{equation*}
f\underset{p}{\sim}-x^2-yz=f_{p}.
\end{equation*}
We have $d_{f_p}=1,m_{f_p}=1$, then $v_p(d_{f_p})=0,v_p(m_{f_p})=0$ and $f_p$ is isotropic. Since $v_p(d_{f_p}),v_p(m_{f_p})$ are $GL_3(\mathbb{Z}_p)$-invariants of ternary quadratic forms \cite[p.4]{LM03}, we have
\begin{equation*}
d_f=N^2,m_f=N,N_f=4N.
\end{equation*}
The genus which $f$ belongs to is denoted by $G_{4N,N^2,\Nodd}$.

Let $\mathcal{O}\subset Q_{\Nodd}$ be an Eichler order of level $2N/\Nodd$. We have
\begin{equation*}
f\underset{2}{\sim}-x^2-2yz=f_{2}.
\end{equation*}
It follows that $d_{f_{2}}=4,m_{f_{2}}=4$, and $v_2(d_{f_2})=2,v_2(m_{f_2})=2$. We can see that $f_p$ is isotropic and
\begin{equation*}
d_f=4N^2,m_f=4N,N_f=4N.
\end{equation*}
The genus which $f$ belongs to is denoted by $G_{4N,4N^2,\Nodd}$.

Let $\mathcal{O}\subset Q_{2\Neven}$ be an Eichler order of level $N/\Neven$. We have
\begin{equation*}
\mathcal{O}_2\cong\mathbb{Z}_2+\frac{1+i}{2}\mathbb{Z}_2+j\mathbb{Z}_2+\frac{1+i}{2}j\mathbb{Z}_2,
\end{equation*}
where $i^2=-3,j^2=2$, and
\begin{equation*}
f\underset{2}{\sim}3x^2-2y^2-2z^2-2yz\underset{2}{\sim}x^2-3y^2-3z^2\underset{2}{\sim}x^2+5y^2+5z^2\underset{2}{\sim}x^2+y^2+z^2.
\end{equation*}
It follows that $d_{f_2}=4,m_{f_2}=4$, and $v_2(d_{f_2})=2,v_2(m_{f_2})=2$. We can see $f_2$ is anisotropic. In such a case, we have
\begin{equation*}
d_f=4N^2,m_f=4N,N_f=4N.
\end{equation*}
The genus which $f$ belongs to is denoted by $G_{4N,4N^2,2\Neven}$.
\end{proof}
In a positive definite quaternion algebra $Q_N$, we choose a complete set of representatives $\{\mathcal{O}_{\mu}\}_{\mu=1,2...T_{N,F}}$ for these types of Eichler orders of level $F$. Then we have the map $M_0$ as follows:
\begin{align*}
M_0:\{\mathcal{O}_{\mu}\}_{\mu=1,2...T_{\Nodd,N/\Nodd}} & \to  G_{4N,N^2,\Nodd}\\
\mathcal{O}_{\mu}\subset Q_{\Nodd} & \mapsto f_{\mathcal{O}_\mu^0} \in G_{4N,N^2,\Nodd};
\end{align*}
\begin{align*}
M_0:\{\mathcal{O}_{\mu}\}_{\mu=1,2...T_{\Nodd,2N/\Nodd}} & \to  G_{4N,4N^2,\Nodd}\\
\mathcal{O}_{\mu}\subset Q_{\Nodd} & \mapsto f_{\mathcal{O}_\mu^0} \in G_{4N,4N^2,\Nodd};
\end{align*}
\begin{align*}
M_0:\{\mathcal{O}_{\mu}\}_{\mu=1,2...T_{2\Neven,N/\Neven}} & \to  G_{4N,4N^2,2\Neven}\\
\mathcal{O}_{\mu}\subset Q_{2\Neven} & \mapsto f_{\mathcal{O}_\mu^0} \in G_{4N,4N^2,2\Neven}.
\end{align*}
By Propoition \ref{OQE} and Propoition \ref{OQD}, $M_0$ is well-defined.

The proof of following proposition is similar to that of Proposition \ref{OQD}.
\begin{prop}\label{SQD} 
Let $N$ be a product of s distinct odd primes, and for an order $\mathcal{O}$ we denote $f_{S^{0}}$ by $f$. 
\par (1)Let $\mathcal{O}\subset Q_{\Nodd}$ be an Eichler order of level $N/\Nodd$, then $d_f=16N^2$, $N_f=4N$ and $f$ is anisotropic only in the $p$-adic field for $p\mid\Nodd$. The genus which $f$ belongs to is denoted by $G_{4N,16N^2,\Nodd}$.
\par (2)Let $\mathcal{O}\subset Q_{\Nodd}$ be an Eichler order of level $2N/\Nodd$, then $d_f=64N^2$, $N_f=8N$ and $f$ is anisotropic only in the $p$-adic field for $p\mid\Nodd$. The genus which $f$ belongs to is denoted by $G_{8N,64N^2,\Nodd}$.
\par (3)Let $\mathcal{O}\subset Q_{2\Neven}$ be an Eichler order of level $N/\Neven$, then $d_f=64N^2$, $N_f=8N$ and $f$ is anisotropic only in the $p$-adic field for $p\mid2\Neven$. The genus which $f$ belongs to is denoted by $G_{8N,64N^2,2\Neven}$.
\end{prop}

Recallling Proposition \ref{OQTC}, Watson transformation $\lambda_4$ is a bijection between $C(8N,64N^2/N_r)$(resp. $C(4N,16N^2/N_r)$) and  $C(4N,4N^2/N_r)$(resp. $C(4N,N^2/N_r)$). We have the following proposition.

\begin{prop}\label{CD1}
Let $\mathcal{O}$ be an Eichler order, then $\lambda_4(f_{S^{0}})=f_{\mathcal{O}^{0}}$.
\end{prop}
\begin{proof}
Let
\begin{align*}
f_{S^{0}} & =\n(x(2\alpha_1)+y(2\alpha_2)+z(2\alpha_3-1))\\
& =4\n(\alpha_1)x^2+4\n(\alpha_2)y^2+(4\n(\alpha_3)-1)z^2+4\tr(\alpha_2\overline{\alpha_3})yz+4\tr(\alpha_1\overline{\alpha_3})xz+4\tr(\alpha_1\overline{\alpha_2})xy,
\end{align*}
and
\begin{align*}
f_{\mathcal{O}^{0}} & =\n(x\alpha_1+y\alpha_2+z(2\alpha_3-1))\\
& =\n(\alpha_1)x^2+\n(\alpha_2)y^2+(4\n(\alpha_3)-1)z^2+2\tr(\alpha_2\overline{\alpha_3})yz+2\tr(\alpha_1\overline{\alpha_3})xz+\tr(\alpha_1\overline{\alpha_2})xy.
\end{align*}
We have
$$M_{f_{S^{0}}}=
\begin{pmatrix}
2 & 0& 0\\
0 & 2& 0\\
0 & 0& 1
\end{pmatrix}
M_{f_{\mathcal{O}^{0}}}
\begin{pmatrix}
2 & 0& 0\\
0 & 2& 0\\
0 & 0& 1
\end{pmatrix}.$$
Let $m=4$, it follows that $f_{S^{0}}(x,y,z)\equiv -z^2\pmod4$. It is not hard to check
\begin{equation*}
M=
\begin{pmatrix}
1 & 0& 0\\
0 & 1& 0\\
0 & 0& 2
\end{pmatrix}
\end{equation*}
and
\begin{align*}
\lambda_4(f_{S^{0}}(x,y,z)) & =\n(\alpha_1)x^2+\n(\alpha_2)y^2+(4\n(\alpha_3)-1)z^2+2\tr(\alpha_2\overline{\alpha_3})yz+2\tr(\alpha_1\overline{\alpha_3})xz+\tr(\alpha_1\overline{\alpha_2})xy\\
& = f_{\mathcal{O}^{0}}(x,y,z).
\end{align*}.
\end{proof}

\begin{cor}
Let $\mathcal{O},\mathcal{O}'\subset Q$ be orders. Suppose $\mathcal{O},\mathcal{O}'$ are isomorphic (resp. locally isomorphic), then $f_{S^{0}},f_{S'^{0}}$ are equivalent (resp. semi-equivalent).
\end{cor}

In a positive definite quaternion algebra $Q_N$, we choose a complete set of representatives $\{\mathcal{O}_{\mu}\}_{\mu=1,2...T_{N,F}}$ for these types of Eichler orders of level $F$. Then we have the map $M_1$ as follows:
\begin{align*}
M_1:\{\mathcal{O}_{\mu}\}_{\mu=1,2...T_{\Nodd,N/\Nodd}} & \to  G_{4N,16N^2,\Nodd}\\
\mathcal{O}_{\mu}\subset Q_{\Nodd} & \mapsto f_{S_\mu^0} \in G_{4N,16N^2,\Nodd};
\end{align*}
\begin{align*}
M_1:\{\mathcal{O}_{\mu}\}_{\mu=1,2...T_{\Nodd,2N/\Nodd}} & \to  G_{8N,64N^2,\Nodd}\\
\mathcal{O}_{\mu}\subset Q_{\Nodd} & \mapsto f_{S_\mu^0} \in G_{8N,64N^2,\Nodd};
\end{align*}
\begin{align*}
M_1:\{\mathcal{O}_{\mu}\}_{\mu=1,2...T_{2\Neven,N/\Neven}} & \to  G_{8N,64N^2,2\Neven}\\
\mathcal{O}_{\mu}\subset Q_{2\Neven} & \mapsto f_{S_\mu^0} \in G_{8N,64N^2,2\Neven}.
\end{align*}
It follows that $M_1$ is  well-defined. We will show $M_0$ and $M_1$ are bijections in Section 6.
\subsection{Even Clifford algebras}
If $f$ is a non-degenerate ternary quadratic form integral over $\mathbb{Z}$, we define $C_0(f)$ to be the even Clifford algebras over $\mathbb{Z}$ associated to $f$. Then $C_0(f)$ is an order in a quaternion algebra over $\mathbb{Q}$. For more detailed definition
of Clifford algebra one can see \cite{Lem11} and \cite{LM03}. Conversely, assume that $
\mathcal{O}=\mathbb{Z}+\mathbb{Z}e_1+\mathbb{Z}e_2+\mathbb{Z}e_3$ is an order in a quaternion algebra $Q$. Then $\Lambda =\mathcal{O}^{\sharp}\cap Q^0$ is a 3-dimensional $\mathbb{Z}$-lattice on $Q^0$. If
$\mathcal{O}^{\sharp}= \langle f_0, f_1, f_2, f_3\rangle$ where $f_i$ is the dual basis of $e_i$, then $\Lambda=\langle f_1, f_2, f_3\rangle$. Define a ternary quadratic form $f_{\mathcal{O}}$ associated to $\mathcal{O}$ by
\begin{equation*}
f_{\mathcal{O}}=\discrd(\mathcal{O})\cdot n(xf_1+yf_2+zf_3).
\end{equation*}
\begin{theorem}\label{Lem}\cite[Theorem 4.3]{Lem11}
Let $R$ be a principal ideal domain. The maps $f\mapsto C_0(f)$ and $\mathcal{O}\mapsto f_{\mathcal{O}}$ are inverses to each other and the discriminants satisfy
\textnormal{discrd}$(\mathcal{O})=d(f_{\mathcal{O}})$. Furthermore, the maps give a bijection between similarity classes of non-degenerate ternary quadratic forms integral over $R$ and isomorphism classes of quaternion $R$-orders.
\end{theorem}
Assume that $f=(a,b,c,r,s,t), d_f=d$. We have
\begin{equation*}
i^2=re_1-bc,jk=a\overline{i}
\end{equation*}
\begin{equation*}
j^2=se_2-ac,ki=b\overline{j}
\end{equation*}
\begin{equation*}
k^2=te_3-ab,ij=c\overline{k}
\end{equation*}
By the definetion of dual basis, it follows that $\mathcal{O}^{\sharp}\supset\mathcal{O}$ and $\tr(f_0)=1, \tr(f_1)=\tr(f_2)=\tr(f_3)=0$, where 
\begin{center}
$df_0=d-2(abc+rst)+(ar+st)i+(bs+rt)j+(ct+rs)k$,
\end{center}
\begin{center}
$df_1=ar+st-2ai-tj-sk$,
\end{center}
\begin{center}
$df_2=bs+rt-ti-2bj-rk$,
\end{center}
\begin{center}
$df_3=ct+rs-si-rj-2ck$,
\end{center}
and 
\begin{center}
$\n(Nf_1)=Na,\tr(Nf_2\overline{Nf_3})=Nr$,
\end{center}
\begin{center}
$\n(Nf_2)=Nb,\tr(Nf_3\overline{Nf_1})=Ns$,
\end{center}
\begin{center}
$\n(Nf_3)=Nc,\tr(Nf_1\overline{Nf_2})=Nt$.
\end{center}
Recalling $\mathcal{O}$ contains a basis of $Q$ over $\mathbb{Q}$, it follows that $\langle f_1, f_2, f_3\rangle$ form a $\mathbb{Q}-$basis for the trace-zero elements of $Q$ (that is, $\langle i,j,k\rangle$ can be represented by $\langle f_1, f_2, f_3\rangle$ over $\mathbb{Q}$). It is straightforward to check that $f_{\mathcal{O}}$ is positive definite if and only if $\mathcal{O}$ is positive definite. We also note that $$2\textnormal{card}(\Aut(\mathcal{O}))=|\Aut(f_{\mathcal{O}})|,$$ and $\mathcal{O}$ ramifies at $p$ if and only if $f_{\mathcal{O}}$ is anisotropic at $p$.

\begin{prop}\label{CQD} 
Let $N$ be a product of s distinct odd primes, and for an order $\mathcal{O}$ we denote $f_{\mathcal{O}}$ by $f$. 
\par (1)Let $\mathcal{O}\subset Q_{\Nodd}$ be an Eichler order of level $N/\Nodd$, then $d_f=N, N_f=4N$ and $f$ is anisotropic only in the $p$-adic field for $p\mid\Nodd$. The genus which $f$ belongs to is denoted by $G_{4N,N,\Nodd}$.
\par (2)Let $\mathcal{O}\subset Q_{\Nodd}$ be an Eichler order of level $2N/\Nodd$, then $d_f=2N, N_f=8N$ and $f$ is anisotropic only in the $p$-adic field for $p\mid\Nodd$. The genus which $f$ belongs to is denoted by $G_{8N,2N,\Nodd}$.
\par (3)Let $\mathcal{O}\subset Q_{2\Neven}$ be an Eichler order of level $N/\Neven$, then $d_f=2N, N_f=8N$ and $f$ is anisotropic only in the $p$-adic field for $p\mid2\Neven$. The genus which $f$ belongs to is denoted by $G_{8N,2N,2\Neven}$.
\end{prop}
\begin{proof}
By Lemma \ref{Leh1}, the level of a positive definite ternary quadratic form with discriminant $N$(resp. $2N$) is $4N$(resp. $8N$).
\end{proof}

\section{Proof of Theorem 1.1}

Now, we   give  prove the Theorem 1.1 in this section.

Based on Lehman's results (Theorem \ref{Leh2} and Lemma \ref{Leh3}), we can easily generalize the same procedure that constructed automorphs or equivalences involving $\phi$ for some $f$ to the $p$-adic integers, thus making it feasible to compute $\phi_p(f)$. This facilitates a straightforward calculation of $\phi_p(f)$.
 Additionally, Proposition \ref{ISO} dictates that for $f \in G_{4N,N^2,\Nodd}$, we have $\phi_2(f) \in G_{4N,16N^2,\Nodd}$.    Assmue that
\begin{equation*}
f=(a,b,c,r,2s,2t),\phi_2(f)=(a,4b,4c,4r,4s,4t)
\end{equation*}
where $2\nmid ac$, then
\begin{center}
$f\underset{2}{\sim}-x^2-yz$, $\phi_2(f)\underset{2}{\sim}-x^2-4yz$.
\end{center}
and for $p\mid\Nodd$, we can choose $u$ such that $\left(\frac{u}{p}\right)=-1$ and $2\nmid u$, then
\begin{center}
$f\underset{p}{\sim}-ux^2-py^2+upz^2$, $\phi_2(f)\underset{p}{\sim}-ux^2-4py^2+4upz^2\underset{p}{\sim}-ux^2-py^2+upz^2$.
\end{center}
\begin{center}
$f\underset{p}{\sim}-x^2-pyz$, $\phi_2(f)\underset{p}{\sim}-x^2-4pyz\underset{p}{\sim}-x^2-pyz$, for $p\mid N/\Nodd$
\end{center}
and 
\begin{center}
$f\underset{p}{\sim}-x^2-yz$, $\phi_2(f)\underset{p}{\sim}-x^2-4yz\underset{p}{\sim}-x^2-yz$, for $p\nmid N$.
\end{center}
By Proposition \ref{Leh2} and Proposition \ref{ISO}, let $N_r=p_1...p_r$, then we have the following bijections:
\begin{align*}
\phi_{p_1}\circ...\circ\phi_{p_r} :G_{4N,N^2,\Nodd} & \to G_{4N,N^2/N_r,\Nodd}\\
f & \mapsto \phi_{p_1}\circ...\circ\phi_{p_r}(f);
\end{align*}
\begin{align*}
\phi_{p_1}\circ...\circ\phi_{p_r} :G_{4N,4N^2,\Nodd} & \to G_{4N,4N^2/N_r,\Nodd}\\
f & \mapsto \phi_{p_1}\circ...\circ\phi_{p_r}(f);
\end{align*}\begin{align*}
\phi_{p_1}\circ...\circ\phi_{p_r} :G_{4N,4N^2,2\Neven} & \to G_{4N,4N^2/N_r,2\Neven}\\
f & \mapsto \phi_{p_1}\circ...\circ\phi_{p_r}(f);
\end{align*}
\begin{align*}
\phi_{2} :G_{4N,N^2/N_r,\Nodd} & \to G_{4N,16N^2/N_r,\Nodd}\\
f & \mapsto \phi_2(f).
\end{align*}
It is the same as the above to calculate $\phi_p(f)$. There are $2^{2s+1}$ genera when $\Nodd,\Neven$ and $N_r$ run over all factors which divides $N$. Noting that there are $2^{s-1}$ (resp. $2^{s}$) genera in $C(4N,N^2)$ (resp. $C(4N,4N^2)$)\cite[Lemma 3]{Leh92}. By Proposition \ref{m1}, there are $2^{2s+1}$ genera in the set of all primitive positive definite ternary quadratic forms of level $4N$. For simplicity of presentation, let $p\mid\Nodd$ (resp. $p\mid\Neven$) and $q\mid N/\Nodd$ (resp. $q\mid N/\Neven$), and we denote the genera as follows:

\begin{table}[H]
\centering
\caption{Genera}
\begin{tabular}{cccc}
\toprule[2pt]
    Genus&2-adic&$p$-adic,$\left(\frac{u}{p}\right)=-1$&$q$-adic\\
\midrule   
   $G_{4N,N^2/N_r,\Nodd}$&$-N_rx^2-yz$&$upN_r^{-1}x^2-pN_r^{-1}y^2-uN_rz^2$&$-N_rx^2-qN_r^{-1}yz$\\
   $G_{4N,4N^2/N_r,\Nodd}$&$-N_rx^2-2yz$&$upN_r^{-1}x^2-pN_r^{-1}y^2-uN_rz^2$&$-N_rx^2-qN_r^{-1}yz$\\
   $G_{4N,4N^2/N_r,2\Neven}$&$N_r^{-1}x^2+N_r^{-1}y^2+N_rz^2$&$upN_r^{-1}x^2-pN_r^{-1}y^2-uN_rz^2$&$-N_rx^2-qN_r^{-1}yz$\\ 
$G_{4N,16N^2/N_r,\Nodd}$&$-N_rx^2-4yz$&$upN_r^{-1}x^2-pN_r^{-1}y^2-uN_rz^2$&$-N_rx^2-qN_r^{-1}yz$\\  
\bottomrule[2pt]     
\end{tabular}
\end{table}

\section{Proof of Theorem 1.2 and Theorem 1.3}

Now we will discuss the connection among $M_0, M_1$ and Clifford algebra using $\lambda_4$ and $\phi$.

\begin{prop}\label{CD2}
(1)Let $\discrd(\mathcal{O})=N$, where $N=p_1...p_s$, then $\phi_{p_1}\circ...\circ\phi_{p_s}(f_{\mathcal{O}^{0}})=f_{\mathcal{O}}$.\\
(2)Let $\discrd(\mathcal{O})=2N$, where $N=p_1...p_s$, then $\phi_{p_1}\circ...\circ\phi_{p_s}\circ\phi_{2}(f_{S^{0}})=f_{\mathcal{O}}$.
\end{prop}
\begin{proof}
Suppose that $N=p$. By Proposition \ref{CQD}, assume $f=(pa,b,c,r,ps,pt)$ be a positive definite ternary quadratic form of level $4p$ and discriminant $p$. Recalling $\phi_p$:
\begin{equation*}
\phi_p^{-1}((pa,b,c,r,ps,pt))=(a,pb,pc,pr,ps,pt),p\nmid ac.
\end{equation*}
Then disc$(\mathcal{O})=p$, and
\begin{equation*}
\mathcal{O}=\mathbb{Z}+\mathbb{Z}i+\mathbb{Z}j+\mathbb{Z}k.
\end{equation*} 
with
\begin{center}
$i^2=ri-bc\qquad jk=pa\overline{i}$
\end{center}
\begin{center}
$j^2=psj-pac\qquad ki=b\overline{j}$
\end{center}
\begin{center}
$k^2=ptk-pab\qquad ij=c\overline{k}$.
\end{center}
Then
\begin{center}
$f_0=1-2(abc+prst)+(ar+pst)i+(bs+rt)j+(ct+rs)k$,
\end{center}
\begin{center}
$f_1=ar+pst-2ai-tj-sk$,
\end{center}
\begin{center}
$f_2=bs+rt-ti-\frac{2b}{p}j-\frac{r}{p}k$,
\end{center}
\begin{center}
$f_3=ct+rs-si-\frac{r}{p}j-\frac{2c}{p}k$,
\end{center}
We have $\tr(f_1)=\tr(f_2)=\tr(f_3)=0$, and 
\begin{center}
$\n(f_1)=a,\tr(pf_2\overline{pf_3})=pr$,
\end{center}
\begin{center}
$\n(pf_2)=pb,\tr(pf_3\overline{f_1})=ps$,
\end{center}
\begin{center}
$\n(pf_3)=pc,\tr(f_1\overline{pf_2})=pt$.
\end{center}
We will show that
\begin{equation*}
\mathcal{O}=\mathbb{Z}f_0+\mathbb{Z}f_1+\mathbb{Z}pf_2+\mathbb{Z}pf_3.
\end{equation*} 
We have
\begin{equation*}
\begin{pmatrix}
f_0\\
f_1\\
pf_2\\
pf_3
\end{pmatrix}
=
\begin{pmatrix}
1-2(abc+prst) & ar+pst & bs+rt & ct+rs\\
ar+pst & -2a & -t & -s\\
p(bs+rt) & -pt& -2b & -r\\
p(ct+rs) & -ps & -r & -2c
\end{pmatrix}
\begin{pmatrix}
1\\
i\\
j\\
k
\end{pmatrix}
=
M_p
\begin{pmatrix}
1\\
i\\
j\\
k
\end{pmatrix}
\end{equation*}
Since $d_f=p$, we have 
\begin{equation*}
4abc+prst-ar^2-pbs^2-pct^2=1.
\end{equation*}
It is not hard to check that
\begin{equation*}
\det(M_p)=-2(4abc+prst-ar^2-pbs^2-pct^2)+(4abc+prst-ar^2-pbs^2-pct^2)^2=-1.
\end{equation*}
Hence
\begin{equation*}
\mathcal{O}=\mathbb{Z}f_0+\mathbb{Z}f_1+\mathbb{Z}pf_2+\mathbb{Z}pf_3,
\end{equation*} 
where $\tr(f_1)=\tr(pf_2)=\tr(pf_3)=0$, and
\begin{align*}
f_{\mathcal{O}^{0}} & =\n(xf_1+ypf_2+zpf_3)\\
& =\n(f_1)x^2+\n(pf_2)y^2+\n(pf_3)z^2+\tr(pf_2\overline{pf_3})yz+\tr(f_1\overline{pf_3})xz+\tr(f_1\overline{pf_2})xy\\
& =ax^2+pby^2+pcz^2+pryz+psxz+ptxy.
\end{align*}
It is similar for $\discrd(\mathcal{O})=2$. Let $f=(2a,b,c,r,4s,4t)$ be a positive definite ternary quadratic forms of level $8$ and discriminant $2$. Recalling $\phi_2$:
\begin{equation*}
\phi_2^{-1}((2a,b,c,r,4s,4t))=(a,8b,8c,8r,8s,8t),2\nmid ac.
\end{equation*}
Then
\begin{equation*}
\mathcal{O}=\mathbb{Z}f_0+\mathbb{Z}f_1+\mathbb{Z}2f_2+\mathbb{Z}2f_3,
\end{equation*} 
\begin{align*}
f_{\mathcal{O}^{0}} & =\n(xf_1+ypf_2+zpf_3)\\
& =\n(f_1)x^2+\n(2f_2)y^2+\n(2f_3)z^2+\tr(2f_2\overline{2f_3})yz+\tr(f_1\overline{2f_3})xz+\tr(f_1\overline{2f_2})xy\\
& =ax^2+2by^2+2cz^2+2ryz+4sxz+4txy.
\end{align*}
We have
\begin{equation*}
f_{S^{0}}=\lambda_4^{-1}(f_{\mathcal{O}^{0}})=ax^2+8by^2+8cz^2+8ryz+8sxz+8txy. 
\end{equation*}
For $\discrd(\mathcal{O})=N$ or $\discrd(\mathcal{O})=2N$, the proof is similar to the above by Remake \ref{rem1}. 
\end{proof}
By Proposition \ref{CD1} and Proposition \ref{CD2}, we give the following connection among $M_0, M_1$ and Clifford algebras using $\lambda_4$ and $\phi$.
\begin{theorem}
Let $N=p_1...p_s$, and $\phi_N$(resp. $\phi_{2N}$) denote $\phi_{p_1}\circ...\circ\phi_{p_s}$(resp. $\phi_{p_1}\circ...\circ\phi_{p_s}\circ\phi_{2}$). In a positive definite quaternion algebra $Q_{\Nodd}$ (resp. $Q_{2\Neven}$), we choose a complete set of representatives $\{\mathcal{O}_{\mu}\}_{\mu=1,2...T_{\Nodd,F}}$ (resp. $\{\mathcal{O}_{\mu}\}_{\mu=1,2...T_{2\Neven,F}}$) for these types of Eichler orders of level $F$. We give the commutative diagrams as follows.
\begin{equation*}
\xymatrix{
    \{\mathcal{O}_\mu\}_{\mu=1,2...T_{\Nodd,N/\Nodd}} \ar[r]^{C_0} \ar[d]^{M_1} \ar[dr]^{M_0} & G_{4N,N,\Nodd} \ar[d]^{\phi^{-1}_N}\\
    G_{4N,16N^2,\Nodd} \ar[r]^{\lambda_4}  & G_{4N,N^2,\Nodd}
}
\end{equation*}
\begin{equation*}
\xymatrix{
    \{\mathcal{O}_\mu\}_{\mu=1,2...T_{\Nodd,2N/\Nodd}} \ar[r]^{C_0} \ar[d]^{M_0} \ar[dr]^{M_1} & G_{8N,2N,\Nodd} \ar[d]^{\phi^{-1}_{2N}}\\
    G_{4N,4N^2,\Nodd}  & G_{8N,64N^2,\Nodd} \ar[l]^{\lambda_4} 
}
\end{equation*}
\begin{equation*}
\xymatrix{
    \{\mathcal{O}_\mu\}_{\mu=1,2...T_{2\Neven,N/\Neven}} \ar[r]^{C_0} \ar[d]^{M_0} \ar[dr]^{M_1} & G_{8N,2N,2\Neven} \ar[d]^{\phi^{-1}_{2N}}\\
    G_{4N,4N^2,2\Neven} & G_{8N,64N^2,2\Neven} \ar[l]^{\lambda_4}  
}
\end{equation*}
\end{theorem}

Since $C_0$ is a bijection, $2\textnormal{card}(\Aut(\mathcal{O}))=|\Aut(f_{\mathcal{O}})|$
and $\lambda_4$, $\phi$ are bijections of equivalence classes which preserve automorphism counts,
it is not hard to prove the following corollary.
\begin{cor}\label{M}
$M_0$ and $M_1$ are bijections, and  
$$2\textnormal{card}(\Aut(\mathcal{O}))=|\Aut(f_{\mathcal{O}^0})|=|\Aut(f_{S^0})|.$$
\end{cor}

\begin{cor}\label{T}
Let $|G|$ denote the number of classes in genus $G$, then
\begin{equation*}
|G_{4N,16N^2/N_r,\Nodd}|=|G_{4N,N^2/N_r,\Nodd}|=T_{\Nodd,N/\Nodd},
\end{equation*}
\begin{equation*}
|G_{4N,4N^2/N_r,\Nodd}|=T_{\Nodd,2N/\Nodd},
\end{equation*}
\begin{equation*}
|G_{4N,4N^2/N_r,2\Neven}|=T_{2\Neven,N/\Neven}.
\end{equation*}
\end{cor}
\subsection{Proof of Theorem 1.3}
By Corollary \ref{T} we have
\begin{equation*}
|C(4N)|=2^s(2\sum\limits_{\Nodd\mid N}T_{\Nodd,N/\Nodd}+\sum\limits_{\Nodd\mid N}T_{\Nodd,2N/\Nodd}+\sum\limits_{\Neven\mid N}T_{2\Neven,N/\Neven}).
\end{equation*}
Li, Skoruppa and the second author \cite[Corollary 1.2]{LSZ22} proved the following type number formula for Eichler order,
\begin{align*}
T_{N,F}&=2^{-e(NF)-1}\sum_{n\mid NF}\sum_{\substack{n\mid r\\r^2\leq 4n}}H^{(N,F)}(4n-r^2)\\
& = 2^{-e(NF)-1}(H^{(N,F)}(4)+2H^{(N,F)}(3)+2H^{(N,F)}(0)+\left(H^{(N,F)}(8)+2H^{(N,F)}(4)\right)\left(1-\left(\frac{N}{2}\right)^2\right)\\
&+\left(H^{(N,F)}(12)+2H^{(N,F)}(3)\right)\left(1-\left(\frac{N}{3}\right)^2\right)+\sum_{\substack{n\geq5\\n\mid NF}}H^{(N,F)}(4n)).
\end{align*}
Let
\begin{equation*}
A(n)=2^{-s}\sum_{\Nodd\mid N}H^{(\Nodd,N/\Nodd)}(n)+3\cdot2^{-s-2}\left(\sum_{\Nodd\mid N}H^{(\Nodd,2N/\Nodd)}(n)+\sum_{\Neven\mid N}H^{(2\Neven,N/\Neven)}(n)\right),
\end{equation*}
\begin{equation*}
B(n)=2^{-s}\sum_{\Nodd\mid N}H^{(\Nodd,N/\Nodd)}(n)+2^{-s-2}\left(\sum_{\Nodd\mid N}H^{(\Nodd,2N/\Nodd)}(n)+\sum_{\Neven\mid N}H^{(2\Neven,N/\Neven)}(n)\right),
\end{equation*}
and
\begin{equation*}
C(n)=2^{-s-2}\left(\sum_{\Nodd\mid N}H^{(\Nodd,2N/\Nodd)}(n)+\sum_{\Neven\mid N}H^{(2\Neven,N/\Neven)}(n)\right).
\end{equation*}
If $3\nmid N$, then
\begin{equation*}
|C(4N)|=2^{s}\left(A(4)+2\cdot B(3)+B(0)+C(8)+\sum_{\substack{{d\mid N}\\d\neq1}}B(4d)+\sum_{\substack{{d\mid N}\\d\neq1}}C(8d)\right).
\end{equation*}
If $3\mid N$, then
\begin{align*}
|C(4N)|&=2^{s}\left(A(4)+4\cdot B(3)+B(0)+C(8)+B(12)+\sum_{\substack{{d\mid N}\\d\neq1,d\neq 3}}B(4d)+\sum_{\substack{{d\mid N}\\d\neq1}}C(8d)\right)\\
&=2^{s}\left(2\cdot B(3)+A(4)+2\cdot B(3)+B(0)+C(8)+\sum_{\substack{{d\mid N}\\d\neq1}}B(4d)+\sum_{\substack{{d\mid N}\\d\neq1}}C(8d)\right).
\end{align*}
We have
\begin{equation*}
A(4)=\frac{5}{8}-\frac{1}{4}\left(\frac{-4}{N}\right),
\end{equation*}
\begin{equation*}
B(0)=\frac{N}{6}-\frac{1}{8},2\cdot B(3)=\frac{1}{2}-\frac{1}{6}\left(\frac{-3}{N}\right),B(4d)=\frac{3}{4}H(4d),
\end{equation*}
\begin{equation*}
C(8)=\frac{1}{4},C(8d)=\frac{1}{4}H(8d).
\end{equation*}
Hence
\begin{align*}
|C(4N)|&=2^s(2\sum\limits_{\Nodd\mid N}T_{\Nodd,N/\Nodd}+\sum\limits_{\Nodd\mid N}T_{\Nodd,2N/\Nodd}+\sum\limits_{\Neven\mid N}T_{2\Neven,N/\Neven})\\
&=2^{s}\left(\frac{N}{6}+\frac{5}{4}-\frac{1}{4}\left(\frac{-4}{N}\right)-\frac{1}{6}\left(\frac{-3}{N}\right)+\frac{1}{2}\left(1-\left(\frac{N}{3}\right)^2\right)+\frac{1}{4}\sum_{\substack{{d\mid N}\\d\neq1}}\left(H(4d)+H(8d)\right)\right).
\end{align*}
\subsection{Proof of Theorem 1.2}
These one-to-one correspondences $\lambda_4$ and $\phi_p$ will also give us some connection of the representation numbers.
\begin{prop}\label{R}
Let $\mathcal{O}$ be an Eichler order, $\rho_{\mathcal{O}}(n,r)$ is the number of zeros of $x^2-rx+n$ in $\mathcal{O}_{\mu}$, then we have
\begin{equation*}
R_{f_{S^{0}}}(n)=
\begin{cases}
R_{f_{\mathcal{O}^{0}}}(n/4) & \text{if}\quad n\equiv0 \pmod4,\\
\rho_{\mathcal{O}}(\frac{n+1}{4},-1) & \text{if}\quad n\equiv3 \pmod4,\\
0 & \text{if}\quad n\equiv1,2\pmod4.\\
\end{cases}
\end{equation*}
\end{prop}
\begin{proof}
Let
\begin{align*}
f_{S^{0}} & =\n(x(2\alpha_1)+y(2\alpha_2)+z(2\alpha_3-1))\\
& =4\n(\alpha_1)x^2+4\n(\alpha_2)y^2+(4\n(\alpha_3)-1)z^2+4\tr(\alpha_2\overline{\alpha_3})yz+4\tr(\alpha_1\overline{\alpha_3})xz+4\tr(\alpha_1\overline{\alpha_2})xy,
\end{align*}
and
\begin{align*}
f_{\mathcal{O}^{0}} & =\n(x\alpha_1+y\alpha_2+z(2\alpha_3-1))\\
& =\n(\alpha_1)x^2+\n(\alpha_2)y^2+(4\n(\alpha_3)-1)z^2+2\tr(\alpha_2\overline{\alpha_3})yz+2\tr(\alpha_1\overline{\alpha_3})xz+\tr(\alpha_1\overline{\alpha_2})xy.
\end{align*}
We have $f_{S^{0}}\equiv-z^2\pmod4$. Hence $R_{f_{S^{0}}}(n)=0$ if $n\equiv1,2\pmod4$.\\ 
If $n\equiv0 \pmod4$, then we have $2\mid z$, and 
\begin{equation*}
n=\n(x(2\alpha_1)+y(2\alpha_2)+z(2\alpha_3-1))=4\n(x\alpha_1+y\alpha_2+\frac{z}{2}(2\alpha_3-1)).
\end{equation*}
We have $x\alpha_1+y\alpha_2+\frac{z}{2}(2\alpha_3-1)\in\mathcal{O}^0$. Conversely, if
\begin{equation*}
\n(x\alpha_1+y\alpha_2+z(2\alpha_3-1))=n/4,
\end{equation*}
then
\begin{equation*}
\n(2x\alpha_1+2y\alpha_2+2z(2\alpha_3-1))=n,
\end{equation*}
and $x(2\alpha_1)+y(2\alpha_2)+2z(2\alpha_3-1)\in S^0$. Hence $R_{f_{S^{0}}}(n)=R_{f_{\mathcal{O}^{0}}}(n/4)$.\\ 
If $n\equiv3 \pmod4$, then $2\nmid z$. Since $\tr(\alpha\overline{\beta})=\n(\alpha+\beta)-\n(\alpha)-\n(\beta)$, we have
\begin{equation*}
\n(x(2\alpha_1)+y(2\alpha_2)+z(2\alpha_3-1))=\n(2x\alpha_1+2y\alpha_2+2z\alpha_3-z+1-1)=\n(2x\alpha_1+2y\alpha_2+2z\alpha_3-z-1)-1.
\end{equation*}
Then $x\alpha_1+y\alpha_2+z\alpha_3-\frac{z+1}{2}\in\mathcal{O}$ and we have
\begin{equation*}
\n(x\alpha_1+y\alpha_2+z\alpha_3-\frac{z+1}{2})=\frac{n+1}{4}.
\end{equation*}
and 
\begin{equation*}
\tr(x\alpha_1+y\alpha_2+z\alpha_3-\frac{z+1}{2})=-1.
\end{equation*}
Conversely, if
\begin{equation*}
\n(x\alpha_1+y\alpha_2+z\alpha_3+t)=\frac{n+1}{4}.
\end{equation*}
and 
\begin{equation*}
\tr(x\alpha_1+y\alpha_2+z\alpha_3+t)=-1.
\end{equation*}
We have $\tr(z\alpha_3)+2t=2t+z=-1$, hence $t=-\frac{z+1}{2}$, and
\begin{equation*}
n=4\n(x\alpha_1+y\alpha_2+z\alpha_3-\frac{z+1}{2})-1=\n(x(2\alpha_1)+y(2\alpha_2)+z(2\alpha_3-1)).
\end{equation*}
Hence $R_{f_{S^{0}}}(n)=\rho_{\mathcal{O}}(\frac{n+1}{4},-1).$
\end{proof}
\begin{prop}\label{m2}
Let $f\in C(N',p^2d')$ where $p$ is an odd prime, $p\parallel N'$ and $p\nmid d'$, then we have
\begin{equation*}
R_{f}(pn)=R_{\phi_p(f)}(n).
\end{equation*}
\end{prop}
\begin{proof}
Let $f \in C(N',p^2d')$. By Lemma \ref{Leh3}, assume that
\begin{center}
$f=(pa,pb,c,pr,ps,pt)$,
\end{center}
with $a, b, c, r, s$ and t integers, $p\nmid ac$. It follows that
\begin{center}
$\phi_p(f)=(a,b,pc,pr,ps,t).$
\end{center}
Suppose that there are integers $x,y$ and $z$ satisfying the equation
\begin{center}
$f(x,y,z)=pax^2+pby^2+cz^2+pryz+psxz+ptxy=pn$.
\end{center}
Since all the terms are divisible by $p$ expect $cz^2$, and $p\nmid c$, we have $p\mid z$. Let
$(x',y',z')=(x,y,z/p)$, then
\begin{align*}
pn=f(x,y,z) &=pax^2+pby^2+cz^2+pryz+psxz+ptxy\\
&=pax'^2+pby'^2+c(pz')^2+pry'(pz')+psx'(pz')+ptx'y'\\
&=p(ax'^2+by'^2+pcz'^2+pry'z'+psx'z'+tx'y')\\
&=p\phi_p(f)(x',y',z').
\end{align*}
We have $\phi_p(f)(x',y',z')=n$. It follows that if there exists integers $(x,y,z)$ such that $f(x,y,z)=pn$, then $\phi_p(f)(x,y,z/p)=n$.

Conversely, 
\begin{center}
$\phi_p(f)=(a,b,pc,pr,ps,t)\sim(pa,b,c,r,ps,pt)$.
\end{center}
Since we only take a single representative from each equivalence class of form, without loss of generality, let
\begin{center}
$\phi_p(f)=(pa,b,c,r,ps,pt)$.
\end{center}
Then we have
\begin{center}
$\phi_p^{-1}(\phi_p(f))=(a,pb,pc,pr,ps,pt)\sim(pa,pb,c,pr,ps,pt)=f.$
\end{center}
Suppose that there are some integers $x,y$ and $z$ such that
\begin{center}
$\phi_p(f)(x,y,z)=pax^2+by^2+cz^2+ryz+psxz+ptxy=n$,
\end{center}
Let $(x',y',z')=(px,y,z)$, then
\begin{align*}
pn=p\phi_p(f)(x,y,z)&=p^2ax^2+pby^2+pcz^2+pryz+p^2sxz+p^2txy\\
&=ax'^2+pby'^2+pcz'^2+pry'z'+psx'z'+ptx'y'\\
&=\phi_p^{-1}(\phi_p(f))(x',y',z').
\end{align*}
Hence $R_{f}(pn)=R_{\phi_p(f)}(n)$.
\end{proof}
Now we will give proof of Theorem 1.2. By Corollary \ref{M}, we have
\begin{equation*}
\sum\limits_{f\in G_{4N,N^2,\Nodd}}\frac{R_f(n)}{|\Aut(f)|}=2^{-s-1}H^{(\Nodd,N/\Nodd)}(4n).
\end{equation*}
By Proposition \ref{m1}, and Proposition \ref{m2}, we have
\begin{equation*}
\sum\limits_{f\in G_{4N,N^2/N_r,\Nodd}}\frac{R_f(n)}{|\Aut(f)|}=2^{-s-1}H^{(\Nodd,N/\Nodd)}(4N_rn).
\end{equation*}
It is the same way to check that
\begin{equation*}
\sum\limits_{f\in G_{4N,4N^2,\Nodd}}\frac{R_f(n)}{|\Aut(f)|}=2^{-s-2}H^{(\Nodd,N/\Nodd)}(4n),
\end{equation*}
\begin{equation*}
\sum\limits_{f\in G_{4N,4N^2/N_r,\Nodd}}\frac{R_f(n)}{|\Aut(f)|}=2^{-s-2}H^{(\Nodd,2N/\Nodd)}(4N_rn),
\end{equation*}
\begin{equation*}
\sum\limits_{f\in G_{4N,4N^2,2\Neven}}\frac{R_f(n)}{|\Aut(f)|}=2^{-s-2}H^{(2\Neven,N/\Neven)}(4n),
\end{equation*}
and
\begin{equation*}
\sum\limits_{f\in G_{4N,4N^2/N_r,2\Neven}}\frac{R_f(n)}{|\Aut(f)|}=2^{-s-2}H^{(2\Neven,N/\Neven)}(4N_rn).
\end{equation*}
By Proposition \ref{R}, we have
\begin{equation*}
\sum\limits_{f\in G_{4N,16N^2/N_r,\Nodd}}\frac{R_f(n)}{|\Aut(f)|}=2^{-s-1}H^{(\Nodd,N/\Nodd)}(N_rn).
\end{equation*}

\section{Some applications}
In this section, we will give three applications of our main results. Construct a basis of Eisenstein space of modular forms of weight 3/2, and give new proofs of  Berkovich and Jagy’s genus identity and Du's identity.

 \subsection{Modular forms of weight 3/2}
 Let $\mathscr{E}(4N, \frac{3}{2}, \chi_l)$ be the Eisenstein series space of modular forms of weight $3/2$, level $4N$ and character $\chi_l$, which is the
 orthogonal complement of the subspace of cusp forms  of weight $3/2$, level $4N$ and the character $\chi_l$ with respect to Petersson inner product. Pei  \cite{WP12} gave an explicit basis of the Eisenstein space $\mathscr{E}(4N, \frac{3}{2}, \chi_l)$.
 
 \begin{theorem}\cite[Theorem 7.7]{WP12}   For any prime $p$, let $h_p(m)$ be the non-negative integer such that $p^{h_p(m)}\parallel m$ and $h'_p(m)=\frac{m}{p^{h_p(m)}}$ the $p'$-part of $m$. Define
 	\begin{equation*}
 		\alpha(m)=
 		\begin{cases}
 			3\cdot 2^{-\frac{1+h_2(m)}{2}} & \text{if} \quad h_2(m) \quad \text{is odd},\\
 			3\cdot 2^{-1-\frac{h_2(m)}{2}} & \text{if} \quad h_2(m) \quad \text{is even and} \quad h_2'(m)\equiv1\pmod4,\\
 			2^{-\frac{h_2(m)}{2}} & \text{if} \quad h_2(m) \quad \text{is even and} \quad h_2'(m)\equiv3\pmod8,\\
 			0 & \text{if} \quad h_2(m) \quad \text{is even and} \quad  h_2'(m)\equiv7\pmod8;
 		\end{cases}
 	\end{equation*}
 	\begin{equation*}
 		\lambda(n,4D)=\lambda_3(n,4D)=L_{4D}(2,\mbox{id.})^{-1}L_{4D}(1,\chi_{-n})\beta_3(n,0,\chi_D,4D).
 	\end{equation*}

 	and
 	\begin{equation*}
 		A(p,m)=
 		\begin{cases}
 			p^{-1}-(1+p)p^{-\frac{3+h_p(m)}{2}} & \text{if} \quad h_p(m) \quad \text{is odd},\\
 			p^{-1}-2p^{-1-\frac{h_p(m)}{2}} & \text{if} \quad h_2(m) \quad \text{is even and} \quad \left(\frac{-h'_p(m)}{p}\right)=-1,\\
 			p^{-1} & \text{if} \quad h_2(m) \quad \text{is even and} \quad \left(\frac{-h'_p(m)}{p}\right)=1;
 		\end{cases}
 	\end{equation*}
 	where $p$ is an odd prime. Let
 	\begin{equation*}
 		g(\chi_l,m,4D)=2\pi l^{1/2}\sum^\infty_{n=1}\lambda(ln,4D)\prod_{p\mid m}(A(p,ln)-p^{-1})n^{1/2}e^{2\pi inz},
 	\end{equation*}
 	\begin{equation*}
 		g(\chi_l,4m,4D)=2\pi l^{1/2}\sum^\infty_{n=1}\lambda(ln,4D)\alpha(ln)\prod_{p\mid m}(A(p,ln)-p^{-1})n^{1/2}e^{2\pi inz},
 	\end{equation*}
 	and 
 	\begin{equation*}
 		g(\chi_l,4D,4D)=1+2\pi l^{1/2}\sum^\infty_{n=1}\lambda(ln,4D)\alpha(ln)\prod_{p\mid D}(A(p,ln)-p^{-1})n^{1/2}e^{2\pi inz},
 	\end{equation*}
 	where $D$ is an odd squarefree integer, $m,l$ and $\beta$ positive divisor of $D$. Then $g(\chi_l,m,4D) (m\mid D)$ and $g(\chi_l,4m,4D)(m\mid D, m\neq1)$ form a basis of $\mathscr{E}(4N, \frac{3}{2}, \chi_l)$.
 \end{theorem}
 
 Denote the generating functions of average representation number of ternary quadratic forms over each genus as following:
 \begin{equation}\label{def:thetaNo}
 	\theta_{G_{4N,N^2,\Nodd}}(z):=2^{s+1}\sum\limits_{n=0}^{\infty}\left(\sum\limits_{f\in G_{4N,N^2,\Nodd}}\frac{R_f(n)}{|\Aut(f)|}\right)q^n=\sum\limits_{n=0}^{\infty}H^{(\Nodd,N/\Nodd)}(4n)q^n,
 \end{equation}
 \begin{equation}\label{def:theta4No}
 	\theta_{G_{4N,4N^2,\Nodd}}(z):=2^{s+2}\sum\limits_{n=0}^{\infty}\left(\sum\limits_{f\in G_{4N,4N^2,\Nodd}}\frac{R_f(n)}{|\Aut(f)|}\right)q^n=\sum\limits_{n=0}^{\infty}H^{(\Nodd,2N/\Nodd)}(4n)q^n,
 \end{equation}
 and
 \begin{equation}\label{def:thetaNe}
 	\theta_{G_{4N,4N^2,2\Neven}}(z):=2^{s+2}\sum\limits_{n=0}^{\infty}\left(\sum\limits_{f\in G_{4N,4N^2,2\Neven}}\frac{R_f(n)}{|\Aut(f)|}\right)q^n=\sum\limits_{n=0}^{\infty}H^{(2\Neven,N/\Neven)}(4n)q^n.
 \end{equation}
 It is well-known that these functions are in the space of Eisenstein series of weight $3/2$, $\mathscr{E}(4N, \frac{3}{2}, \text{id})$.

 \begin{theorem}\label{theo:mod-the}
 	Let $I$ denote the set of all positive divisor of $ 2N $ except 1, $d\in I$ . Set 
 	\begin{equation}\label{def:thetaNd}
 		\theta_{d,2N/d}(z):=\sum\limits_{n=0}^{\infty}H^{(d,2N/d)}(4n)q^n.
 	\end{equation} 
 	Let $l$ be the divisor of $N$, and $\chi_l$ denote the primitive characters such that $\chi_l(n)=\left(\frac{l}{n}\right)$ for $(n,4l)=1$. 
 	Then the set  $\{\theta_{d,2N/d}(lz)\}_I$  is a basis of space of Eisenstein series of weight $3/2$ with character $\chi_l$,  $\mathscr{E}(4N, \frac{3}{2}, \chi_l)$.
 \end{theorem}

 \begin{proof}
 	Since $N$ is squarefree with $s$ distinct odd prime factors, it is well-known that $\text{dim} \mathscr{E}(4N, \frac{3}{2}, \chi_l)=2^{s+1}-1$. To give a basis of $ \mathscr{E}(4N, \frac{3}{2}, \chi_l)$, it is suffices to find $2^{s+1}-1$ linear independent elements in the space $\mathscr{E}(4N, \frac{3}{2}, \chi_l)$.
 	  
 	We only prove the  case $\chi_l=\text{id}$, that is $l=1$, and the others can be proved similarly.
 	
 	First, we show that  $\theta_{d,2N/d}(z) \in \mathscr{E}(4N, \frac{3}{2}, \text{id}) $ for all $d\in I$.
 	
 	When the divisor $d$ of $2N$ has only odd number of prime factors, 
 	\begin{equation*}
 		\theta_{d,2N/d}(z)=\theta_{G_{4N,4N^2,d}}(z).
 	\end{equation*}
 		Hence $\theta_{d,2N/d}(z)\in\mathscr{E}(4N, \frac{3}{2}, \text{id})$.
 		
 		When the divisor $d$ of $2N$ has even number  of prime factors, 
 in terms of 
 	\begin{equation}\label{Hurwitz-class-num-def2}
 		H^{(N_1,N_2)}(D)=\sum\limits_{n\mid N_2}\mu(n)2^{e(N_2/n)}H^{(N_1n,1)}(D),
 	\end{equation}
 	where $-D$ is a negative discriminant\cite[p.379]{LSZ22},  we have
 	\begin{equation*}
 		H^{(d/p,2Np/d)}(D)=-H^{(d,2N/d)}(D)+2H^{(d/p,2N/d)}(D),
 	\end{equation*}
 	where $p$ is some  prime divisor of $d$. It follows that
 	\begin{equation*}
 		\theta_{d,2N/d}(z)=2\theta_{d/p,2N/d}(z)-\theta_{d/p,2Np/d}(z),
 	\end{equation*}
 	\begin{equation*}
 		\theta_{d/p,2Np/d}(z)\in\mathscr{E}(4N, \frac{3}{2}, \text{id}),
 	\end{equation*}
 	and
 	\begin{equation*}
 		\theta_{d/p,2N/d}(z)\in\mathscr{E}(4N/p, \frac{3}{2}, \text{id})\subset\mathscr{E}(4N, \frac{3}{2}, \text{id}).
 	\end{equation*}
 	Hence $\theta_{d,2N/d}(z)\in\mathscr{E}(4N, \frac{3}{2}, \text{id})$. 
 	
 	Now we show that theta series $\theta_{d,2N/d}(z), d\in I$,  are linearly independent. Suppose  
 	\begin{equation*}
 		\sum\limits_{ d \in I}c(d)\theta_{d,2N/d}(z)=0.
 	\end{equation*}
 	  We show that coefficients $c(d), d\in I$ vanish by induction. Assume that the divisor $d\in I$ is prime. By Chinese remainder theorem, choose a discriminant $-D_d<0$ such that $\left(\frac{-D_d}{d}\right)=-1$ and $\left(\frac{-D_d}{p}\right)=1$, where $p\mid 2N/d$. Then we have 
 	$H^{(d,2N/d)}(4D_d)\neq0, $ and $ H^{(d',2N/d')}(4D_d)=0	$ for $d'\in I, d'\neq d$. 
  It follows that if $d$ is prime, $c(d)=0$.
 	Similarly, assume that $d$ is a product of two primes, $d=pq$, where $p,q\mid 2N$ are prime. If $p=2$ or $q=2$, without loss of generality we assume $q=2$. Choose a discriminant $-D_{2p}<0$ such that $\left(\frac{-D_{2p}}{p}\right)=\left(\frac{-D_{2p}}{2}\right)=-1$ and $\left(\frac{-D_{pq}}{p'}\right)=1$, where $p'\mid 2N/(pq)$.  Then we have 
 	\begin{equation*}
 		H^{(p,2N/p)}(4D_{pq})\neq0, H^{(2p,N/p)}(4D_{pq})\neq0,
 	\end{equation*}
 	and the others equal zero. It follows that $c(2p)=0$. If $2\nmid pq$, choose a discriminant $-D_{pq}<0$ such that $\left(\frac{-D_{pq}}{p}\right)=\left(\frac{-D_{pq}}{q}\right)=-1$ and $\left(\frac{-D_{pq}}{p'}\right)=1$, where $p'\mid 2N/(pq)$.  Then we have 
 	\begin{equation*}
 		H^{(pq,2N/(pq))}(4D_{pq})\neq0,
 	\end{equation*}
 	and the others equal zero. It follows that $c(pq)=0$.  
 	
 	By induction, for all $d\in I$,  coefficients $c(d)$ vanish. 
 	So we prove that  the set  $\{\theta_{d,2N/d}(z)\}_I$  is a basis of $\mathscr{E}(4N, \frac{3}{2}, \text{id})$.
 	
 \end{proof}

\subsection{Berkovich and Jagy’s genus identity}
 
 Berkovich and Jagy \cite{BJa12} established the following interesting identity connecting the weighted sum of the representation numbers and the sum of three squares $r_3(n)$:
\begin{equation}\label{eq:Berk-Jagy}
r_3(p^2n)-pr_3(n)=48\sum\limits_{f\in TG_{1,p}}\frac{R_{f}(n)}{|\Aut(f)|}-96\sum\limits_{f\in TG_{2,p}}\frac{R_{f}(n)}{|\Aut(f)|},
\end{equation}
where a sum over forms in a genus should be understood to be the finite sum resulting from taking a single representative from each equivalence class of forms. 
 
 We now give a new proof of the identity (\ref{eq:Berk-Jagy}).
By the level and discriminant of ternary quadratic forms,  it becomes evident  that $TG_{1,p}$ (resp. $TG_{2,p}$) coincides $G_{4p,p^2,p}$ (resp. $G_{4p,16p^2,p}$). 
In terms of the results of Theorem \ref{theo:rep}, we have
 \begin{equation}\label{eq:BJ1}
 	 \sum\limits_{f\in TG_{1,p}}\frac{R_{f}(n)}{|\Aut(f)|}=\sum\limits_{f\in G_{4p,p^2,p}}\frac{R_{f}(n)}{|\Aut(f)|}=\frac{1}{4}H^{(p,1)}(4n);
 \end{equation}
and
\begin{equation}\label{eq:BJ2}
	\sum\limits_{f\in TG_{2,p}}\frac{R_{f}(n)}{|\Aut(f)|}=\sum\limits_{f\in G_{4p,16p^2,p}}\frac{R_{f}(n)}{|\Aut(f)|}=\frac{1}{4}H^{(p,1)}(n).
\end{equation}

For a squarefree integer $M$ and an odd prime $p$ with $(M,p)=1$ and a negative discriminant $-D$, one  \cite{BSZ19} has  

\begin{equation}\label{eq:Hurwitz-class-max}
	H^{(M,1)}(p^2D)-pH^{(M,1)}(D)=H^{(pM,1)}(D).
\end{equation}
 
So we have 
\begin{equation}\label{eq:BJ3}
H^{(2,1)}(4p^2n)-pH^{(2,1)}(4n)=H^{(2p,1)}(4n)
\end{equation}
 and 
 \begin{equation}
H^{(p,1)}(16n)-2H^{(p,1)}(4n)=H^{(2p,1)}(4n).
\end{equation}
Furthermore, from the above equation we can get 

\begin{equation}\label{eq:BJ4}
	H^{(p,1)}(4n)-2H^{(p,1)}(n)=H^{(2p,1)}(4n).
\end{equation}
 Combining the eqatuions (\ref{eq:BJ1}), (\ref{eq:BJ2}),  (\ref{eq:BJ3}), (\ref{eq:BJ4}) and $r_3(n)=12H^{(2,1)}(4n)$, we obtain the identity (\ref{eq:Berk-Jagy}).
  
  In fact, combining Theorem \ref{theo:rep} and the equality of the modified Hurwitz class number $H^{N_1,N_2}(D)$, we can derive more identities similar to (\ref{eq:Berk-Jagy}). 

\subsection{Du's equality}
In 2016, Du \cite{Du16} gave an interesting equality as follows. For squarefree $D$ with odd number of prime factors, let $B(D)$ be the unique quaternion algebra over $\mathbb{Q}$ of discriminant $D$, and 
\begin{equation*}
V(D)=\{x\in B(D)| \text{tr}(x)=0\}.
\end{equation*}
For a positive interger $N$ prime to $D$, let $L_D(N)=\mathcal{O}(N)\cap V(D)$, where $\mathcal{O}(N)$ is an Eichler order in $B$ of conductor $N$, and 
\begin{equation*}
r_{D,N}=r_{\text{gen}(L)}(m)=\frac{\sum\limits_{L_1\in \text{gen}(L)}\frac{r_{L_1}(m)}{|\Aut(L_1)|}}{\sum\limits_{L_1\in \text{gen}(L)}\frac{1}{|\Aut(L_1)|}}.
\end{equation*}
Now let $D$ be a square-free positive integer with even number of prime factors, $p\neq q$ be two different primes not dividing $D$, and $N$ be a positive integer prime to $Dpq$. Then
\begin{equation}\label{eq:du}
-\frac{2}{q-1}r_{Dp,N}(m)+\frac{q+1}{q-1}r_{Dp,Nq}(m)=-\frac{2}{p-1}r_{Dq,N}(m)+\frac{p+1}{p-1}r_{Dq,Np}(m)
\end{equation}
for every positive integer $m$. 

Now we give  a new proof of the equality (\ref{eq:du}).  In terms of the results of Theorem \ref{theo:rep},

\begin{equation*}
r_{Dp,Nq}(m)=\frac{H^{(Dp,Nq)}(4m)}{H^{(Dp,Nq)}(0)}=\frac{2H^{(Dp,N)}(4m)-H^{(Dpq,N)}(4m)}{(1+q)H^{(Dp,N)}(0)}.
\end{equation*}
In terms of the equality (\ref{Hurwitz-class-num-def2}  ), we have
\begin{equation*}
	H^{(D,Nq)}(4m)=2H^{(D,N)}(4m)-H^{(Dq,N)}(4m).
\end{equation*}
The left hand of the equality (\ref{eq:du}) 
\begin{equation*}
-\frac{2}{q-1}r_{Dp,N}(m)+\frac{q+1}{q-1}r_{Dp,Nq}(m)=\frac{-H^{(Dpq,N)}(4m)}{(q-1)H^{(Dp,N)}(0)}=\frac{H^{(Dpq,N)}(4m)}{H^{(Dpq,N)}(0)}.
\end{equation*}
Similarly, the right hand of the equality (\ref{eq:du})
\begin{equation*}
-\frac{2}{p-1}r_{Dq,N}(m)+\frac{p+1}{p-1}r_{Dq,Np}(m)=\frac{H^{(Dpq,N)}(4m)}{H^{(Dpq,N)}(0)}.
\end{equation*}

\newpage

\section{Examples}

In this section we will give some examples of representation of ternary quadratic forms. If the class number of a genus is one, we can give an exact formula of representation number of $n$ by the ternary quadratic forms. 
For squarefree integers $N$, there are 73 genera of ternary quadratic forms of level $4N$ with one class. 
Let $\mathcal{O}\subset Q_N$ be an Eichler order. Its type number equals 1 if its level $(N,F)$ is one of the following: (2,1), (3,1), (5,1), (7,1), (13,1), (30,1), (42,1), (70,1), (78,1), (2,3), (2,5), (2,7), (2,11), (2,15), (2,23), (3,2), (3,5), (3,11), (5,2), (7,3)\cite[p.94]{Li21}, and we get 73 genera with one class. We give the explicit formulas for the representation number of ternary quadratic forms as follows.
\begin{longtable}{llll}
\caption{Genera  with one class}\\
\toprule[2pt]
Genus&$N_f$&$d_f$&$R_f(n)$\\
\midrule
$G_{4,4,2}$&$4$&$4$&$R_{(1,1,1,0,0,0)}(n)=12H^{(2,1)}(4n)$\\
$G_{12,9,3}$&$4\cdot3$&$3^2$&$R_{(1,1,3,0,0,-1)}(n)=6H^{(3,1)}(4n)$\\
$G_{12,3,3}$&$4\cdot3$&$3$&$R_{(1,1,1,0,0,-1)}(n)=6H^{(3,1)}(12n)$\\
$G_{12,144,3}$&$4\cdot3$&$16\cdot3^2$&$R_{(3,4,4,-4,0,0)}(n)=6H^{(3,1)}(n)$\\
$G_{12,48,3}$&$4\cdot3$&$16\cdot3$&$R_{(1,4,4,-4,0,0)}(n)=6H^{(3,1)}(3n)$\\
$G_{20,25,5}$&$4\cdot5$&$5^2$&$R_{(2,2,2,-1,-1,-1)}(n)=3H^{(5,1)}(4n)$\\
$G_{20,5,5}$&$4\cdot5$&$5$&$R_{(1,1,2,1,1,1)}(n)=3H^{(5,1)}(20n)$\\
$G_{20,400,5}$&$4\cdot5$&$16\cdot5^2$&$R_{(3,7,7,-6,-2,-2)}(n)=3H^{(5,1)}(n)$\\
$G_{20,80,5}$&$4\cdot5$&$16\cdot5$&$R_{(3,3,3,2,2,2)}(n)=3H^{(5,1)}(5n)$\\
$G_{28,49,7}$&$4\cdot7$&$7^2$&$R_{(1,2,7,0,0,-1)}(n)=2H^{(7,1)}(4n)$\\
$G_{28,7,7}$&$4\cdot7$&$7$&$R_{(1,1,2,0,-1,0)}(n)=2H^{(7,1)}(28n)$\\
$G_{28,784,7}$&$4\cdot7$&$16\cdot7^2$&$R_{(4,7,8,0,-4,0)}(n)=2H^{(7,1)}(n)$\\
$G_{28,112,7}$&$4\cdot7$&$16\cdot7$&$R_{(1,4,8,-4,0,0)}(n)=2H^{(7,1)}(7n)$\\
$G_{52,169,13}$&$4\cdot13$&$13^2$&$R_{(2,5,5,-3,-1,-1)}(n)=H^{(13,1)}(4n)$\\
$G_{52,13,13}$&$4\cdot13$&$13$&$R_{(1,2,2,-1,0,-1)}(n)=H^{(13,1)}(52n)$\\
$G_{52,2704,13}$&$4\cdot13$&$16\cdot13^2$&$R_{(7,8,15,8,2,4)}(n)=H^{(13,1)}(n)$\\
$G_{52,208,13}$&$4\cdot13$&$16\cdot13$&$R_{(3,3,7,2,2,2)}(n)=H^{(13,1)}(13n)$\\
$G_{60,900,30}$&$4\cdot3\cdot5$&$4\cdot3^2\cdot5^2$&$2R_{(3,10,10,-10,0,0)}(n)=3H^{(30,1)}(4n)$\\
$G_{60,300,30}$&$4\cdot3\cdot5$&$4\cdot3\cdot5^2$&$2R_{(1,10,10,-10,0,0)}(n)=3H^{(30,1)}(12n)$\\
$G_{60,180,30}$&$4\cdot3\cdot5$&$4\cdot3^2\cdot5$&$2R_{(2,2,15,0,0,-2)}(n)=3H^{(30,1)}(20n)$\\
$G_{60,60,30}$&$4\cdot3\cdot5$&$4\cdot3\cdot5$&$2R_{(2,2,5,0,0,-2)}(n)=3H^{(30,1)}(60n)$\\
$G_{84,1764,42}$&$4\cdot3\cdot7$&$4\cdot3^2\cdot7^2$&$R_{(1,21,21,0,0,0)}(n)=H^{(42,1)}(4n)$\\
$G_{84,588,42}$&$4\cdot3\cdot7$&$4\cdot3\cdot7^2$&$R_{(3,7,7,0,0,0)}(n)=H^{(42,1)}(12n)$\\
$G_{84,252,42}$&$4\cdot3\cdot7$&$4\cdot3^2\cdot7$&$R_{(3,3,7,0,0,0)}(n)=H^{(42,1)}(28n)$\\
$G_{84,84,42}$&$4\cdot3\cdot7$&$4\cdot3\cdot7$&$R_{(1,1,21,0,0,0)}(n)=H^{(42,1)}(84n)$\\
$G_{140,4900,70}$&$4\cdot5\cdot7$&$4\cdot5^2\cdot7^2$&$2R_{(2,18,35,0,0,-2)}(n)=H^{(70,1)}(4n)$\\
$G_{140,980,70}$&$4\cdot5\cdot7$&$4\cdot5\cdot7^2$&$2R_{(6,6,7,0,0,-2)}(n)=H^{(70,1)}(20n)$\\
$G_{140,700,70}$&$4\cdot5\cdot7$&$4\cdot5^2\cdot7$&$2R_{(5,6,6,-2,0,0)}(n)=H^{(70,1)}(28n)$\\
$G_{140,140,70}$&$4\cdot5\cdot7$&$4\cdot5\cdot7$&$2R_{(1,2,18,-2,0,0)}(n)=H^{(70,1)}(140n)$\\
$G_{156,6084,78}$&$4\cdot3\cdot13$&$4\cdot3^2\cdot13^2$&$2R_{(6,13,21,0,-6,0)}(n)=H^{(78,1)}(4n)$\\
$G_{156,2028,78}$&$4\cdot3\cdot13$&$4\cdot3\cdot13^2$&$2R_{(2,7,39,0,0,-2)}(n)=H^{(78,1)}(12n)$\\
$G_{156,468,78}$&$4\cdot3\cdot13$&$4\cdot3^2\cdot13$&$2R_{(1,6,21,-6,0,0)}(n)=H^{(78,1)}(52n)$\\
$G_{156,156,78}$&$4\cdot3\cdot13$&$4\cdot3\cdot13$&$2R_{(2,33,7,0,-2,0)}(n)=H^{(78,1)}(156n)$\\
$G_{12,36,2}$&$4\cdot3$&$4\cdot3^2$&$R_{(2,2,3,0,0,-2)}(n)=3H^{(2,3)}(4n)$\\
$G_{12,12,2}$&$4\cdot3$&$4\cdot3$&$R_{(1,2,2,-2,0,0)}(n)=3H^{(2,3)}(12n)$\\
$G_{20,100,2}$&$4\cdot5$&$4\cdot5^2$&$R_{(1,5,5,0,0,0)}(n)=2H^{(2,5)}(4n)$\\
$G_{20,20,2}$&$4\cdot5$&$4\cdot5$&$R_{(1,1,5,0,0,0)}(n)=2H^{(2,5)}(20n)$\\
$G_{28,196,2}$&$4\cdot7$&$4\cdot7^2$&$2R_{(3,5,5,-4,-2,-2)}(n)=3H^{(2,7)}(4n)$\\
$G_{28,28,2}$&$4\cdot7$&$4\cdot7$&$2R_{(2,2,3,2,2,2)}(n)=3H^{(2,7)}(28n)$\\
$G_{44,484,2}$&$4\cdot11$&$4\cdot11^2$&$R_{(2,6,11,0,0,-2)}(n)=H^{(2,11)}(4n)$\\
$G_{44,44,2}$&$4\cdot11$&$4\cdot11$&$R_{(1,2,6,-2,0,0)}(n)=H^{(2,11)}(44n)$\\
$G_{60,900,2}$&$4\cdot3\cdot5$&$4\cdot3^2\cdot5^2$&$2R_{(5,6,9,-6,0,0)}(n)=H^{(2,15)}(4n)$\\
$G_{60,300,2}$&$4\cdot3\cdot5$&$4\cdot3\cdot5^2$&$2R_{(2,3,15,0,0,-2)}(n)=H^{(2,15)}(12n)$\\
$G_{60,180,2}$&$4\cdot3\cdot5$&$4\cdot3^2\cdot5$&$2R_{(1,6,9,-6,0,0)}(n)=H^{(2,15)}(20n)$\\
$G_{60,60,2}$&$4\cdot3\cdot5$&$4\cdot3\cdot5$&$2R_{(2,3,3,0,0,-2)}(n)=H^{(2,15)}(60n)$\\
$G_{92,2116,2}$&$4\cdot23$&$4\cdot23^2$&$2R_{(5,10,14,10,2,4)}(n)=H^{(2,23)}(4n)$\\
$G_{92,92,2}$&$4\cdot23$&$4\cdot23$&$2R_{(2,3,5,-2,0,-2)}(n)=H^{(2,23)}(92n)$\\
$G_{12,36,3}$&$4\cdot3$&$4\cdot3^2$&$R_{(1,3,3,0,0,0)}(n)=2H^{(3,2)}(4n)$\\
$G_{12,12,3}$&$4\cdot3$&$4\cdot3$&$R_{(1,1,3,0,0,0)}(n)=2H^{(3,2)}(12n)$\\
$G_{60,225,3}$&$4\cdot3\cdot5$&$3^2\cdot5^2$&$R_{(1,4,15,0,0,-1)}(n)=H^{(3,5)}(4n)$\\
$G_{60,75,3}$&$4\cdot3\cdot5$&$3\cdot5^2$&$R_{(2,2,5,0,0,-1)}(n)=H^{(3,5)}(12n)$\\
$G_{60,45,3}$&$4\cdot3\cdot5$&$3^2\cdot5$&$R_{(2,2,3,0,0,-1)}(n)=H^{(3,5)}(20n)$\\
$G_{60,15,3}$&$4\cdot3\cdot5$&$3\cdot5$&$R_{(1,1,4,0,-1,0)}(n)=H^{(3,5)}(60n)$\\
$G_{60,3600,3}$&$4\cdot3\cdot5$&$16\cdot3^2\cdot5^2$&$R_{(4,15,16,0,-4,0)}(n)=H^{(3,5)}(n)$\\
$G_{60,1200,3}$&$4\cdot3\cdot5$&$16\cdot3\cdot5^2$&$R_{(5,8,8,-4,0,0)}(n)=H^{(3,5)}(3n)$\\
$G_{60,720,3}$&$4\cdot3\cdot5$&$16\cdot3^2\cdot5$&$R_{(3,8,8,-4,0,0)}(n)=H^{(3,5)}(5n)$\\
$G_{60,240,3}$&$4\cdot3\cdot5$&$16\cdot3\cdot5$&$R_{(1,4,16,-4,0,0)}(n)=H^{(3,5)}(15n)$\\
$G_{132,1089,3}$&$4\cdot3\cdot11$&$3^2\cdot11^2$&$2R_{(6,7,10,7,3,6)}(n)=H^{(3,11)}(4n)$\\
$G_{132,363,3}$&$4\cdot3\cdot11$&$3\cdot11^2$&$2R_{(2,7,7,3,1,1)}(n)=H^{(3,11)}(12n)$\\
$G_{132,99,3}$&$4\cdot3\cdot11$&$3^2\cdot11$&$2R_{(2,3,5,-3,-1,0)}(n)=H^{(3,11)}(44n)$\\
$G_{132,33,3}$&$4\cdot3\cdot11$&$3\cdot11$&$2R_{(1,2,5,1,1,1)}(n)=H^{(3,11)}(132n)$\\
$G_{132,17424,3}$&$4\cdot3\cdot11$&$16\cdot3^2\cdot11^2$&$2R_{(7,19,39,-18,-6,-2)}(n)=H^{(3,11)}(n)$\\
$G_{132,5808,3}$&$4\cdot3\cdot11$&$16\cdot3\cdot11^2$&$2R_{(8,13,17,2,4,8)}(n)=H^{(3,11)}(3n)$\\
$G_{132,1584,3}$&$4\cdot3\cdot11$&$16\cdot3^2\cdot11$&$2R_{(5,5,17,-2,-2,-2)}(n)=H^{(3,11)}(11n)$\\
$G_{132,528,3}$&$4\cdot3\cdot11$&$16\cdot3\cdot11$&$2R_{(4,7,7,-6,0,-4)}(n)=H^{(3,11)}(33n)$\\
$G_{20,100,5}$&$4\cdot5$&$4\cdot5^2$&$R_{(2,3,5,0,0,-2)}(n)=2H^{(5,2)}(4n)$\\
$G_{20,20,5}$&$4\cdot5$&$4\cdot5$&$R_{(1,2,3,-2,0,0)}(n)=2H^{(5,2)}(20n)$\\
$G_{84,441,7}$&$4\cdot3\cdot7$&$3^2\cdot7^2$&$2R_{(2,8,8,-5,-1,-1)}(n)=H^{(7,3)}(4n)$\\
$G_{84,147,7}$&$4\cdot3\cdot7$&$3\cdot7^2$&$2R_{(3,3,5,-2,-2,-1)}(n)=H^{(7,3)}(12n)$\\
$G_{84,63,7}$&$4\cdot3\cdot7$&$3^2\cdot7$&$2R_{(2,2,5,2,2,1)}(n)=H^{(7,3)}(28n)$\\
$G_{84,21,7}$&$4\cdot3\cdot7$&$3\cdot7$&$2R_{(1,2,3,-1,-1,0)}(n)=H^{(7,3)}(84n)$\\
$G_{84,7056,7}$&$4\cdot3\cdot7$&$16\cdot3^2\cdot7^2$&$2R_{(8,11,23,2,8,4)}(n)=H^{(7,3)}(n)$\\
$G_{84,2352,7}$&$4\cdot3\cdot7$&$16\cdot3\cdot7^2$&$2R_{(5,12,12,-4,-4,-4)}(n)=H^{(7,3)}(3n)$\\
$G_{84,1008,7}$&$4\cdot3\cdot7$&$16\cdot3^2\cdot7$&$2R_{(5,8,8,4,4,4)}(n)=H^{(7,3)}(7n)$\\
$G_{84,336,7}$&$4\cdot3\cdot7$&$16\cdot3\cdot7$&$2R_{(3,3,11,2,2,2)}(n)=H^{(7,3)}(21n)$\\
\bottomrule[2pt]
\end{longtable}

\newpage

For the class number larger than one, we give one example.  Let $N=140=4\cdot5\cdot7$, there are $2^{2\times 2+1}=32$ genera in the set of all primitive positive definite ternary quadratic forms of level $140$.  The class number of primitive positive definite ternary quadratic forms of level 140 is  $|C(140)|=76$.  We give 32 formulas of the weighted sums of representation of ternary quadratic forms on each genus.

\begin{longtable}{llll}
\caption{Representation of  ternary quadratic forms of level 140}\\
\toprule[2pt]
Genus&$d_f$&$R_f(n)$\\
\midrule
$G_{140,1225,5}$&$5^2\cdot7^2$&$3R_{(3,3,35,0,0,-1)}(n)+2R_{(3,12,12,-11,-2,-2)}(n)+3R_{(5,7,10,0,-5,0)}(n)=3H^{(5,7)}(4n)$\\
$G_{140,1225,7}$&$5^2\cdot7^2$&$R_{(1,9,35,0,0,-1)}(n)+2R_{(4,9,11,9,1,2)}(n)=H^{(7,5)}(4n)$\\
$G_{140,245,5}$&$5\cdot7^2$&$3R_{(1,7,9,0,-1,0)}(n)+2R_{(4,4,4,1,1,1)}(n)+3R_{(1,2,35,0,0,-1)}(n)=3H^{(5,7)}(20n)$\\
$G_{140,245,7}$&$5\cdot7^2$&$R_{(3,3,7,0,0,-1)}(n)+2R_{(3,5,5,3,2,2)}(n)=H^{(7,5)}(20n)$\\
$G_{140,175,5}$&$5^2\cdot7$&$3R_{(1,5,10,-5,0,0)}(n)+2R_{(4,4,4,3,3,3)}(n)+3R_{(1,5,9,0,-1,0)}(n)=3H^{(5,7)}(28n)$\\
$G_{140,175,7}$&$5^2\cdot7$&$R_{(3,3,5,0,0,-1)}(n)+2R_{(2,2,12,-1,-1,-1)}(n)=H^{(7,5)}(28n)$\\
$G_{140,35,5}$&$5\cdot7$&$3R_{(1,3,3,-1,0,0)}(n)+2R_{(1,1,12,1,1,1)}(n)+3R_{(1,2,5,0,0,-1)}(n)=3H^{(5,7)}(140n)$\\
$G_{140,35,7}$&$5\cdot7$&$R_{(1,1,9,0,-1,0)}(n)+2R_{(1,3,4,3,1,1)}(n)=H^{(7,5)}(140n)$\\ 
$G_{140,4900,2}$&$4\cdot5^2\cdot7^2$&$2R_{(5,14,21,-14,0,0)}(n)+2R_{(7,10,20,-10,0,0)}(n)=H^{(2,35)}(4n)$\\
$G_{140,4900,5}$&$4\cdot5^2\cdot7^2$&$4R_{(3,12,35,0,0,-2)}(n)+2R_{(5,7,35,0,0,0)}(n)+2R_{(7,10,20,-10,0,0)}(n)=H^{(5,14)}(4n)$\\
$G_{140,4900,7}$&$4\cdot5^2\cdot7^2$&$R_{(1,35,35,0,0,0)}(n)+4R_{(4,9,35,0,0,-2)}(n)+4R_{(11,11,15,10,10,8)}(n)=H^{(7,10)}(4n)$\\
$G_{140,4900,70}$&$4\cdot5^2\cdot7^2$&$2R_{(2,18,35,0,0,-2)}(n)=H^{(70,1)}(4n)$\\
$G_{140,980,2}$&$4\cdot5\cdot7^2$&$2R_{(2,7,18,0,-2,0)}(n)+2R_{(1,14,21,-14,0,0)}(n)=H^{(2,35)}(20n)$\\
$G_{140,980,5}$&$4\cdot5\cdot7^2$&$4R_{(4,7,9,0,-2,0)}(n)+2R_{(1,7,35,0,0,0)}(n)+2R_{(2,4,35,0,0,-2)}(n)=H^{(5,14)}(20n)$\\
$G_{140,980,7}$&$4\cdot5\cdot7^2$&$R_{(5,7,7,0,0,0)}(n)+4R_{(3,5,19,-4,-2,-2)}(n)+4R_{(3,7,12,0,-2,0)}(n)=H^{(7,10)}(20n)$\\
$G_{140,980,70}$&$4\cdot5\cdot7^2$&$2R_{(6,6,7,0,0,-2)}(n)=H^{(70,1)}(20n)$\\
$G_{140,700,2}$&$4\cdot5^2\cdot7$&$2R_{(2,5,18,0,-2,0)}(n)+2R_{(2,3,25,0,0,-2)}(n)=H^{(2,35)}(28n)$\\
$G_{140,700,5}$&$4\cdot5^2\cdot7$&$4R_{(4,5,9,0,-2,0)}(n)+2R_{(1,5,35,0,0,0)}(n)+2R_{(1,10,20,-10,0,0)}(n)=H^{(5,14)}(28n)$\\
$G_{140,700,7}$&$4\cdot5^2\cdot7$&$R_{(5,5,7,0,0,0)}(n)+4R_{(3,5,12,0,-2,0)}(n)+4R_{(2,8,13,6,2,2)}(n)=H^{(7,10)}(28n)$\\
$G_{140,700,70}$&$4\cdot5^2\cdot7$&$2R_{(5,6,6,-2,0,0)}(n)=H^{(70,1)}(28n)$\\
$G_{140,140,2}$&$4\cdot5\cdot7$&$2R_{(1,6,6,-2,0,0)}(n)+2R_{(2,3,7,0,0,-2)}(n)=H^{(2,35)}(140n)$\\
$G_{140,140,5}$&$4\cdot5\cdot7$&$4R_{(1,3,12,-2,0,0)}(n)+2R_{(1,5,7,0,0,0)}(n)+2R_{(2,4,5,0,0,-2)}(n)=H^{(5,14)}(140n)$\\
$G_{140,140,7}$&$4\cdot5\cdot7$&$R_{(1,1,35,0,0,0)}(n)+4R_{(1,4,9,-2,0,0)}(n)+4R_{(3,4,4,-2,-2,-2)}(n)=H^{(7,10)}(140n)$\\
$G_{140,140,70}$&$4\cdot5\cdot7$&$2R_{(1,2,18,-2,0,0)}(n)=H^{(70,1)}(140n)$\\
$G_{140,19600,5}$&$16\cdot5^2\cdot7^2$&$3R_{(7,20,40,20,0,0)}(n)+2R_{(3,47,47,-46,-2,-2)}(n)+3R_{(12,12,35,0,0,4)}(n)=3H^{(5,7)}(n)$\\
$G_{140,19600,7}$&$16\cdot5^2\cdot7^2$&$R_{(4,35,36,0,-4,0)}(n)+2R_{(11,15,39,-10,-6,-10)}(n)=H^{(7,5)}(n)$\\
$G_{140,3920,5}$&$16\cdot5\cdot7^2$&$3R_{(4,7,36,0,-4,0)}(n)+2R_{(11,11,11,-6,-6,-6)}(n)+3R_{(4,8,35,0,0,-4)}(n)=3H^{(5,7)}(5n)$\\
$G_{140,3920,7}$&$16\cdot5\cdot7^2$&$R_{(7,12,12,-4,0,0)}(n)+2R_{(3,19,19,10,2,2)}(n)=H^{(7,5)}(5n)$\\
$G_{140,2800,5}$&$16\cdot5^2\cdot7$&$3R_{(1,20,40,-20,0,0)}(n)+2R_{(9,9,9,-2,-2,-2)}(n)+3R_{(4,5,36,0,-4,0)}(n)=3H^{(5,7)}(7n)$\\
$G_{140,2800,7}$&$16\cdot5^2\cdot7$&$R_{(5,12,12,-4,0,0)}(n)+2R_{(8,8,13,-4,-4,-4)}(n)=H^{(7,5)}(7n)$\\
$G_{140,560,5}$&$16\cdot5\cdot7$&$3R_{(4,5,8,0,-4,0)}(n)+2R_{(4,4,13,4,4,4)}(n)+3R_{(1,12,12,-4,0,0)}(n)=3H^{(5,7)}(35n)$\\
$G_{140,560,7}$&$16\cdot5\cdot7$&$R_{(1,4,36,-4,0,0)}(n)+2R_{(4,5,9,-2,0,-4)}(n)=H^{(7,5)}(35n)$\\
\bottomrule[2pt]
\end{longtable}
\newpage
\section*{Appendix}
\appendix
\section{The modified Hurwitz class number}
\begin{longtable}{cccccccccc}
\caption{The modified Hurwitz class number}\\
\toprule[2pt]
$D$&$H(D)$&$H^{(2,1)}(D)$&$H^{(3,1)}(D)$&$H^{(5,1)}(D)$&$H^{(7,1)}(D)$&$H^{(13,1)}(D)$&$H^{(30,1)}(D)$&$H^{(42,1)}(D)$\\
\midrule
0 & -1/12 & 1/12 & 1/6 & 1/3 & 1/2 & 1 & 2/3 & 1\\
3 & 1/3 & 2/3 & 1/3 & 2/3 & 0 & 0 & 4/3 & 0\\
4 & 1/2 & 1/2 & 1 & 0 & 1 & 0 & 0 & 2\\
7 & 1 & 0 & 2 & 2 & 1 & 2 & 0 & 0\\
8 & 1 & 1 & 0 & 2 & 2 & 2 & 0 & 0\\
11 & 1 & 2 & 0 & 0 & 2 & 2 & 0 & 0\\
12 & 4/3 & 2/3 & 4/3 & 8/3 & 0 & 0 & 4/3 & 0\\
15 & 2 & 0 & 2 & 2 & 4 & 4 & 0 & 0\\
16 & 3/2 & 1/2 & 3 & 0 & 3 & 0 & 0 & 2\\
19 & 1 & 2 & 2 & 0 & 0 & 2 & 0 & 0\\
20 & 2 & 2 & 0 & 2 & 0 & 4 & 0 & 0\\
23 & 3 & 0 & 0 & 6 & 6 & 0 & 0 & 0\\
24 & 2 & 2 & 2 & 0 & 0 & 4 & 0 & 0\\
27 & 4/3 & 8/3 & 1/3 & 8/3 & 0 & 0 & 4/3 & 0\\
28 & 2 & 0 & 4 & 4 & 2 & 4 & 0 & 0\\
31 & 3 & 0 & 6 & 0 & 0 & 6 & 0 & 0\\
32 & 3 & 1 & 0 & 6 & 6 & 6 & 0 & 0\\
35 & 2 & 4 & 0 & 2 & 2 & 0 & 0 & 0\\
36 & 5/2 & 5/2 & 1 & 0 & 5 & 0 & 0 & 2\\
39 & 4 & 0 & 4 & 0 & 8 & 4 & 0 & 0\\
40 & 2 & 2 & 4 & 2 & 0 & 0 & 4 & 0\\
43 & 1 & 2 & 2 & 2 & 2 & 0 & 8 & 8\\
44 & 4 & 2 & 0 & 0 & 8 & 8 & 0 & 0\\
47 & 5 & 0 & 0 & 10 & 0 & 10 & 0 & 0\\
48 & 10/3 & 2/3 & 10/3 & 20/3 & 0 & 0 & 4/3 & 0\\
51 & 2 & 4 & 2 & 0 & 4 & 0 & 0 & 8\\
52 & 2 & 2 & 4 & 4 & 0 & 2 & 8 & 0\\
55 & 4 & 0 & 8 & 4 & 0 & 0 & 0 & 0\\
56 & 4 & 4 & 0 & 0 & 4 & 0 & 0 & 0\\
59 & 3 & 6 & 0 & 0 & 0 & 6 & 0 & 0\\
60 & 4 & 0 & 4 & 4 & 8 & 8 & 0 & 0\\
63 & 5 & 0 & 2 & 10 & 5 & 10 & 0 & 0\\
64 & 7/2 & 1/2 & 7 & 0 & 7 & 0 & 0 & 2\\
67 & 1 & 2 & 2 & 2 & 2 & 2 & 8 & 8\\
68 & 4 & 4 & 0 & 8 & 0 & 0 & 0 & 0\\
71 & 7 & 0 & 0 & 0 & 14 & 14 & 0 & 0\\
72 & 3 & 3 & 0 & 6 & 6 & 6 & 0 & 0\\
75 & 7/3 & 14/3 & 7/3 & 2/3 & 0 & 0 & 4/3 & 0\\
76 & 4 & 2 & 8 & 0 & 0 & 8 & 0 & 0\\
79 & 5 & 0 & 10 & 0 & 10 & 0 & 0 & 0\\
80 & 6 & 2 & 0 & 6 & 0 & 12 & 0 & 0\\
83 & 3 & 6 & 0 & 6 & 0 & 6 & 0 & 0\\
84 & 4 & 4 & 4 & 0 & 4 & 8 & 0 & 4\\
87 & 6 & 0 & 6 & 12 & 0 & 0 & 0 & 0\\
88 & 2 & 2 & 4 & 4 & 4 & 0 & 8 & 8\\
91 & 2 & 4 & 4 & 0 & 2 & 2 & 0 & 8\\
92 & 6 & 0 & 0 & 12 & 12 & 0 & 0 & 0\\
95 & 8 & 0 & 0 & 8 & 16 & 0 & 0 & 0\\
96 & 6 & 2 & 6 & 0 & 0 & 12 & 0 & 0\\
99 & 3 & 6 & 0 & 0 & 6 & 6 & 0 & 0\\
100 & 5/2 & 5/2 & 5 & 0 & 0 & 0 & 0 & 10\\
\bottomrule[2pt]
\end{longtable}
\begin{longtable}{cccccccccc}
\caption{The modified Hurwitz class number}\\
\toprule[2pt]
$D$&$H^{(70,1)}(D)$&$H^{(78,1)}(D)$&$H^{(2,3)}(D)$&$H^{(2,5)}(D)$&$H^{(2,7)}(D)$&$H^{(2,11)}(D)$&$H^{(2,15)}(D)$&$H^{(2,23)}(D)$\\
\midrule

0& 2& 2& 1/3& 1/2& 2/3& 1& 2& 2\\
3& 0& 0& 2/3& 0& 4/3& 0& 0& 0\\
4& 0& 0& 0& 1& 0& 0& 0& 0\\
7& 0& 0& 0& 0& 0& 0& 0& 0\\
8& 4& 0& 2& 0& 0& 2& 0& 0\\
11& 0& 0& 4& 4& 0& 2& 8& 4\\
12& 0& 0& 2/3& 0& 4/3& 0& 0& 0\\
15& 0& 0& 0& 0& 0& 0& 0& 0\\
16& 0& 0& 0& 1& 0& 0& 0& 0\\
19& 0& 8& 0& 4& 4& 4& 0& 4\\
20& 0& 0& 4& 2& 4& 0& 4& 4\\
23& 0& 0& 0& 0& 0& 0& 0& 0\\
24& 0& 4& 2& 4& 4& 4& 4& 0\\
27& 0& 0& 14/3& 0& 16/3& 0& 0& 0\\
28& 0& 0& 0& 0& 0& 0& 0& 0\\
31& 0& 0& 0& 0& 0& 0& 0& 0\\
32& 4& 0& 2& 0& 0& 2& 0& 0\\
35& 4& 0& 8& 4& 4& 8& 8& 0\\
36& 0& 0& 4& 5& 0& 0& 8& 0\\
39& 0& 0& 0& 0& 0& 0& 0& 0\\
40& 0& 0& 0& 2& 4& 4& 0& 4\\
43& 8& 0& 0& 0& 0& 4& 0& 4\\
44& 0& 0& 4& 4& 0& 2& 8& 4\\
47& 0& 0& 0& 0& 0& 0& 0& 0\\
48& 0& 0& 2/3& 0& 4/3& 0& 0& 0\\
51& 0& 0& 4& 8& 0& 8& 8& 8\\
52& 0& 4& 0& 0& 4& 4& 0& 0\\
55& 0& 0& 0& 0& 0& 0& 0& 0\\
56& 0& 0& 8& 8& 4& 0& 16& 8\\
59& 0& 0& 12& 12& 12& 0& 24& 0\\
60& 0& 0& 0& 0& 0& 0& 0& 0\\
63& 0& 0& 0& 0& 0& 0& 0& 0\\
64& 0& 0& 0& 1& 0& 0& 0& 0\\
67& 8& 8& 0& 0& 0& 0& 0& 4\\
68& 0& 0& 8& 0& 8& 8& 0& 8\\
71& 0& 0& 0& 0& 0& 0& 0& 0\\
72& 12& 0& 6& 0& 0& 6& 0& 0\\
75& 0& 0& 14/3& 8& 28/3& 0& 8& 0\\
76& 0& 8& 0& 4& 4& 4& 0& 4\\
79& 0& 0& 0& 0& 0& 0& 0& 0\\
80& 0& 0& 4& 2& 4& 0& 4& 4\\
83& 0& 0& 12& 0& 12& 12& 0& 12\\
84& 0& 8& 4& 8& 4& 8& 8& 8\\
87& 0& 0& 0& 0& 0& 0& 0& 0\\
88& 8& 0& 0& 0& 0& 2& 0& 4\\
91& 0& 8& 0& 8& 4& 0& 0& 8\\
92& 0& 0& 0& 0& 0& 0& 0& 0\\
95& 0& 0& 0& 0& 0& 0& 0& 0\\
96& 0& 4& 2& 4& 4& 4& 4& 0\\
99& 0& 0& 12& 12& 0& 6& 24& 12\\
100& 0& 0& 0& 5& 0& 0& 0& 0\\

\bottomrule[2pt]
\end{longtable}
\begin{longtable}{cccccccccc}
\caption{The modified Hurwitz class number}\\
\toprule[2pt]
$D$&$H^{(3,2)}(D)$&$H^{(3,5)}(D)$&$H^{(3,11)}(D)$&$H^{(5,2)}(D)$&$H^{(7,3)}(D)$&$H^{(5,7)}(D)$&$H^{(7,5)}(D)$&$H^{(2,35)}(D)$\\
\midrule

0& 1/2& 1& 2& 1& 2& 8/3& 3& 4\\
3& 0& 0& 0& 0& 0& 4/3& 0& 0\\
4& 1& 2& 0& 0& 0& 0& 2& 0\\
7& 4& 0& 4& 4& 0& 2& 0& 0\\
8& 0& 0& 0& 2& 4& 0& 0& 0\\
11& 0& 0& 0& 0& 4& 0& 4& 0\\
12& 2& 0& 0& 4& 0& 16/3& 0& 0\\
15& 4& 2& 0& 4& 4& 0& 4& 0\\
16& 5& 6& 0& 0& 0& 0& 6& 0\\
19& 0& 4& 4& 0& 0& 0& 0& 8\\
20& 0& 0& 0& 2& 0& 4& 0& 4\\
23& 0& 0& 0& 12& 12& 0& 0& 0\\
24& 2& 4& 4& 0& 0& 0& 0& 8\\
27& 0& 0& 0& 0& 0& 16/3& 0& 0\\
28& 8& 0& 8& 8& 0& 4& 0& 0\\
31& 12& 12& 0& 0& 0& 0& 0& 0\\
32& 0& 0& 0& 10& 12& 0& 0& 0\\
35& 0& 0& 0& 0& 4& 2& 2& 4\\
36& 1& 2& 0& 0& 8& 0& 10& 0\\
39& 8& 8& 8& 0& 8& 0& 16& 0\\
40& 4& 4& 8& 2& 0& 4& 0& 4\\
43& 0& 0& 4& 0& 0& 0& 0& 0\\
44& 0& 0& 0& 0& 16& 0& 16& 0\\
47& 0& 0& 0& 20& 0& 20& 0& 0\\
48& 6& 0& 0& 12& 0& 40/3& 0& 0\\
51& 0& 4& 4& 0& 4& 0& 8& 0\\
52& 4& 0& 8& 4& 0& 8& 0& 0\\
55& 16& 8& 8& 8& 0& 8& 0& 0\\
56& 0& 0& 0& 0& 8& 0& 8& 8\\
59& 0& 0& 0& 0& 0& 0& 0& 24\\
60& 8& 4& 0& 8& 8& 0& 8& 0\\
63& 4& 0& 4& 20& 8& 10& 0& 0\\
64& 13& 14& 0& 0& 0& 0& 14& 0\\
67& 0& 0& 0& 0& 0& 0& 0& 0\\
68& 0& 0& 0& 8& 0& 16& 0& 0\\
71& 0& 0& 0& 0& 28& 0& 28& 0\\
72& 0& 0& 0& 6& 12& 0& 0& 0\\
75& 0& 4& 0& 0& 0& 4/3& 0& 16\\
76& 12& 16& 16& 0& 0& 0& 0& 8\\
79& 20& 20& 20& 0& 0& 0& 20& 0\\
80& 0& 0& 0& 10& 0& 12& 0& 4\\
83& 0& 0& 0& 0& 0& 12& 0& 0\\
84& 4& 8& 8& 0& 4& 0& 8& 8\\
87& 12& 0& 12& 24& 0& 24& 0& 0\\
88& 4& 0& 4& 4& 0& 0& 0& 0\\
91& 0& 8& 0& 0& 0& 0& 4& 8\\
92& 0& 0& 0& 24& 24& 0& 0& 0\\
95& 0& 0& 0& 16& 32& 0& 16& 0\\
96& 10& 12& 12& 0& 0& 0& 0& 8\\
99& 0& 0& 0& 0& 12& 0& 12& 0\\
100& 5& 10& 0& 0& 0& 0& 10& 0\\
\bottomrule[2pt]
\end{longtable}

\newpage
\section{Class number of primitive positive definite ternary quadratic forms of level $4N$}
\begin{longtable}{rr||rr||rr}
	\caption{}\label{tab:class_number1}\\
	\hline
	$4N$&$|C(4N)|$&$4N$&$|C(4N)|$&$4N$&$|C(4N)|$\\
	\hline
	4 & 1     &340 & 108    &668 & 98\\
	12 & 8    &348 & 140    &692 & 86\\
	20 & 8    &356 & 54     &708 & 200\\
	28 & 10     &364 & 124    &716 & 96\\
	44 & 14     &372 & 124    &724 & 86\\
	52 & 12     &380 & 148    &732 & 224\\
	60 & 48     &388 & 50     &740 & 212\\
	68 & 16     &404 & 60     &748 & 212\\
	76 & 18     &412 & 62     &764 & 110\\
	84 & 52     &420 & 384    &772 & 82\\
	92 & 22     &428 & 60     &780 & 600\\
	116 & 22    &436 & 52     &788 & 88\\
	124 & 26    &444 & 164    &796 & 106\\
	132 & 68    &452 & 56     &804 & 224\\
	140 & 76    &460 & 156    &812 & 236\\
	148 & 22    &476 & 180    &820 & 216\\
	156 & 80    &492 & 160    &836 & 244\\
	164 & 30    &508 & 68     &844 & 96\\
	172 & 28    &516 & 164    &852 & 248\\
	188 & 38    &524 & 80     &860 & 276\\
	204 & 92    &532 & 152    &868 & 232\\
	212 & 32    &548 & 66     &876 & 244\\
	220 & 100     &556 & 74     &884 & 236\\
	228 & 88    &564 & 176    &892 & 114\\
	236 & 44    &572 & 200    &908 & 116\\
	244 & 36    &580 & 172    &916 & 106\\
	260 & 96    &596 & 76     &924 & 680\\
	268 & 38    &604 & 80     &932 & 102\\
	276 & 112     &620 & 200    &940 & 252\\
	284 & 50    &628 & 76     &948 & 264\\
	292 & 40    &636 & 216    &956 & 132\\
	308 & 112     &644 & 200    &964 & 110\\
	316 & 48    &652 & 74     &988 & 264\\
	332 & 54    &660 & 496    &996 & 268\\
	
	\hline
\end{longtable}
\section*{Acknowledgements}
This research was supported partially by the National Natural Science Foundation of China(NSFC), Grant no. 12271405. We would like to express our appreciation  to Professor Nils-Peter Skoruppa for his helpful comments on this paper.

\bibliographystyle{plain}
\bibliography{TQF}

\end{document}